\documentclass[reqno]{amsart}
\usepackage{amsaddr}
\usepackage{url}
\usepackage{graphicx}
\usepackage{overpic}
\usepackage{amsmath,amssymb}
\newcommand{\onefig}[1]{
\begin{figure}[htbp]
	\centering
	#1
\end{figure}
}
\newcommand{\henbi}[2]{\frac{\partial #1}{\partial #2}}
\newcommand{\R}{\mathbb R}
\newcommand{\set}[1]{\left\{#1\right\}}
\newcommand{\norm}[1]{\left\lVert#1\right\rVert}

\newcommand{\eps}{\varepsilon}
\newcommand{\abs}[1]{\left\lvert#1\right\rvert}
\newcommand{\mynorm}[1]{\norm{#1}}

\newcommand{\nequal}{\fallingdotseq}
\newcommand{\sobw}[2]{
	\ifx #10
		L^{#2}
	\else 
		\if #22
			H^{#1}
		\else
			W^{#1,#2}
		\fi
	\fi
}
\newcommand{\normsob}[3]{
	\if #32
		\norm{#1}_{#2}
	\else
		\norm{#1}_{#2,#3}
	\fi
}
\newcommand{\snormsob}[3]{
	\if #32
		\abs{#1}_{#2}
	\else
		\abs{#1}_{#2,#3}
	\fi
}
\newcommand{\normi}[1]{\lVert#1\rVert}
\newcommand{\absi}[1]{\lvert#1\rvert}
\newcommand{\normsobi}[3]{
	\if #32
		\normi{#1}_{#2}
	\else
		\normi{#1}_{#2,#3}
	\fi
}
\newcommand{\snormsobi}[3]{
	\if #32
		\absi{#1}_{#2}
	\else
		\absi{#1}_{#2,#3}
	\fi
}
\theoremstyle{plain}
\newtheorem{lemma}{Lemma}
\newtheorem{proposition}{Proposition}
\newtheorem{theorem}{Theorem}
\theoremstyle{definition}
\newtheorem*{Thschemezero}{Scheme LG}
\newtheorem*{Thschemejurai}{Scheme LG$^\prime$}
\newtheorem*{Thschemezettai}{Scheme GSLG}
\newtheorem{hypo}{Hypothesis}
\newtheorem{exam}{Example}
\newtheorem*{acknowledgements}{Acknowledgements}
\theoremstyle{remark}
\newtheorem{remark}{Remark}
\newcommand{\cuzero}{c_0}
\newcommand{\cuone}{c_1}
\newcommand{\cutwo}{c_2}
\newcommand{\cint}{c_{\Pi}}
\newcommand{\cuint}{\alpha_*}
\newcommand{\dtgiven}{\Delta t_0}
\newcommand{\cpro}{c_P}
\newcommand{\done}{\delta_*}

\newcommand{\cprop}{c_*}
\newcommand{\cproptwo}{c_{**}}
\newcommand{\cpropthree}{c_{***}}

\newcommand{\pk}[1]{P_{#1}}
\newcommand{\matrixnorm}[1]{\abs{#1}_{C(W^{1,\infty})}}
\newcommand{\pro}[1]{\widehat{#1}_h}
\newcommand{\eone}{e_h}
\newcommand{\proe}{\eta}
\newcommand{\bwd}{\overline D_{\Delta t}}
\newcommand{\jac}[2]{\abs{\henbi{#1}{#2}}}
\newcommand{\meas}[1]{\operatorname{meas}(#1)}

\newcommand{\schemezero}{Scheme LG}
\newcommand{\jurai}{Scheme LG$^\prime$}
\newcommand{\zettai}{Scheme GSLG}
\newcommand{\markone}{\includegraphics[height=2.5mm]{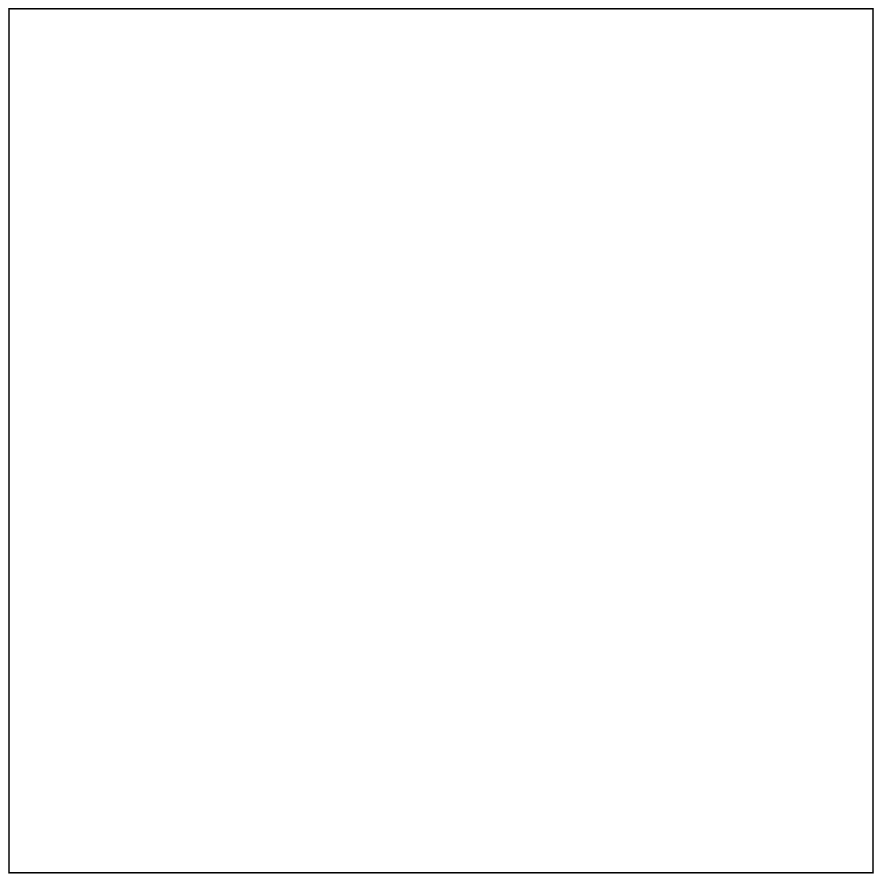}}
\newcommand{\marktwo}{\includegraphics[height=2.5mm]{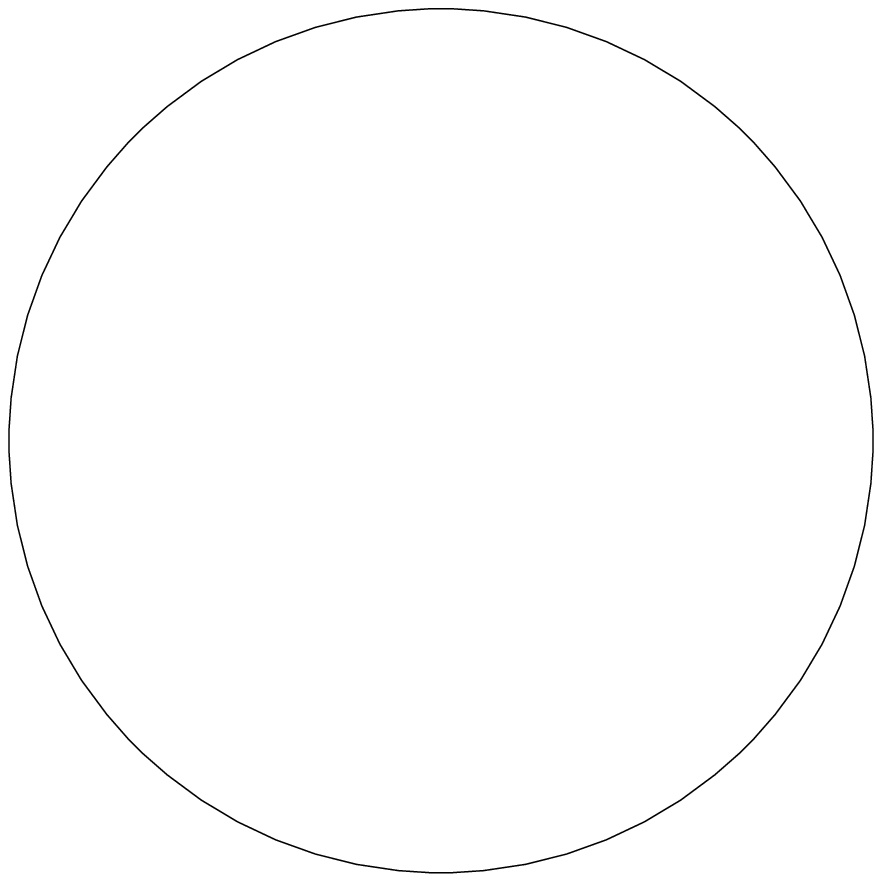}}
\newcommand{\markthree}{\includegraphics[height=2.5mm]{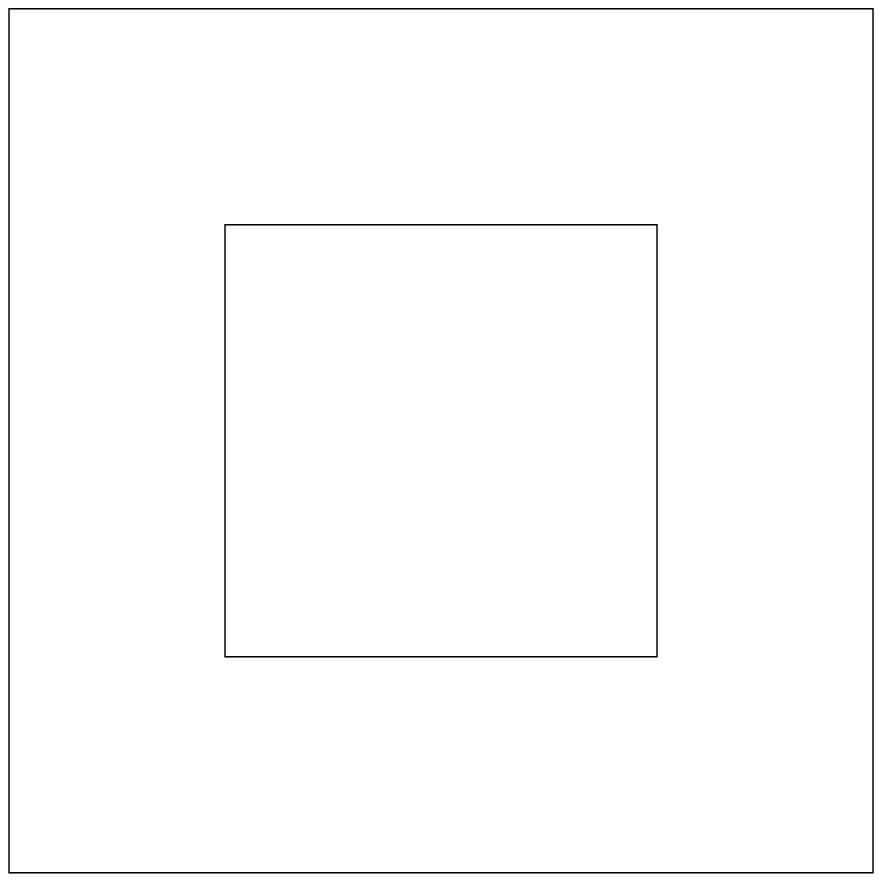}}
\newcommand{\markfour}{\includegraphics[height=2.5mm]{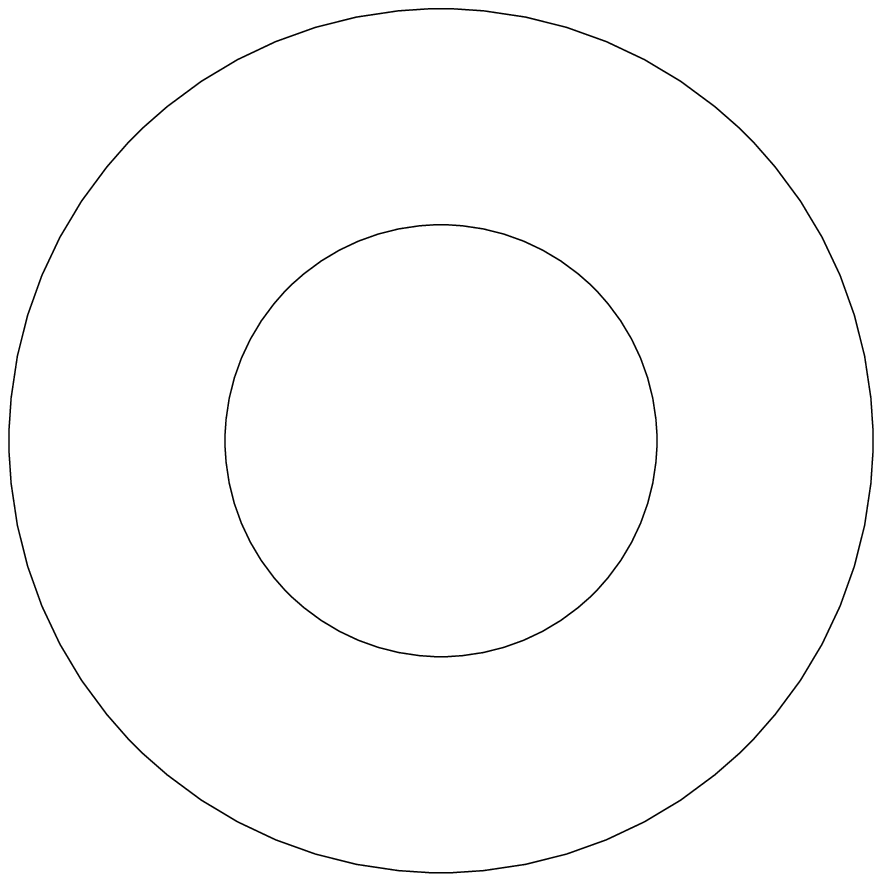}}
\newcommand{\markfive}{\includegraphics[height=2.5mm]{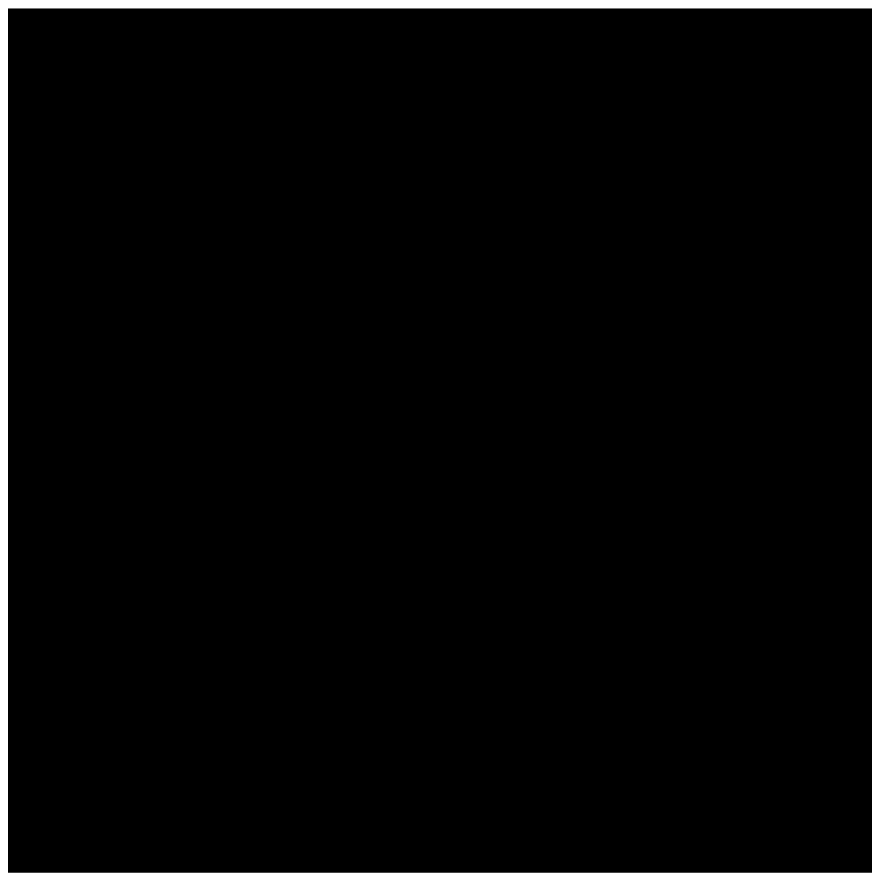}}
\newcommand{\marksix}{\includegraphics[height=2.5mm]{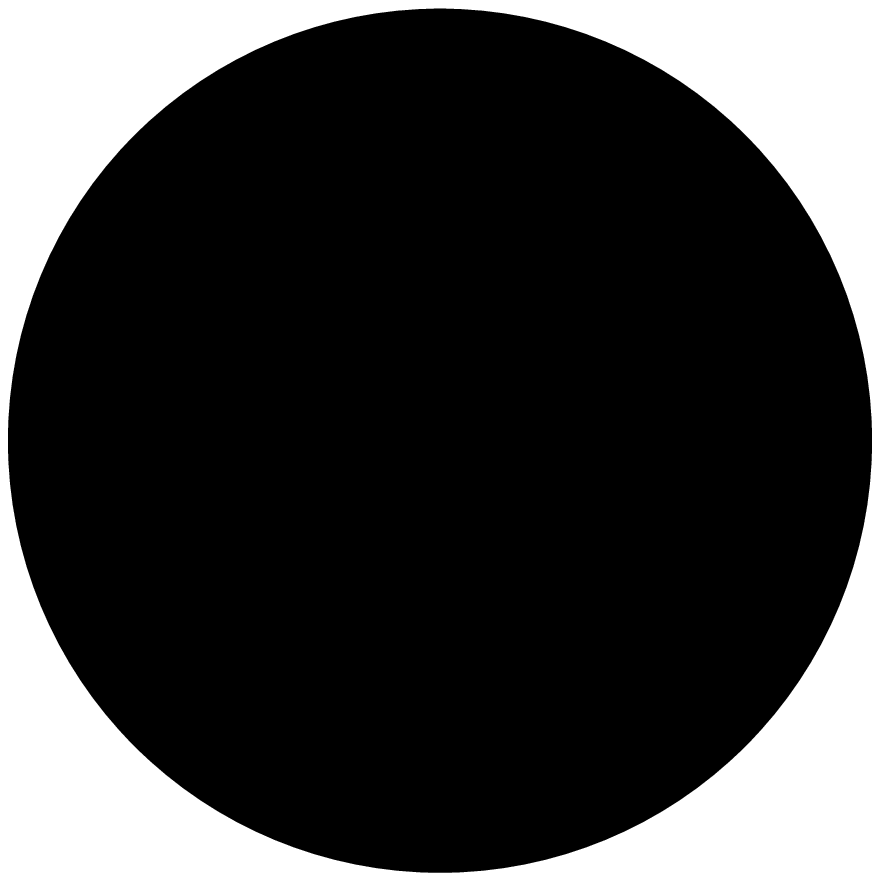}}
\newcommand{\figwidthbasic}{35mm}
\newcommand{\columnheight}{3mm}
\newenvironment{mytable}
	{\begin{table}[h]}
	{\end{table}}
\begin{document}
\title{
A genuinely stable Lagrange--Galerkin scheme for convection-diffusion problems} 
%
\author{Masahisa Tabata} 
\address[M. Tabata]{Department of Mathematics, Waseda University,
	3-4-1, Ohkubo, Shinjuku, Tokyo 169-8555, Japan\\
	\url{tabata@waseda.jp}}
\author{Shinya Uchiumi}
\address[S. Uchiumi]{Research Fellow of Japan Society for the Promotion of Science \\
	Graduate School of Fundamental Science and Engineering, Waseda University,
	3-4-1, Ohkubo, Shinjuku, Tokyo 169-8555, Japan\\
	\url{su48@fuji.waseda.jp}}
%
\date{April 18, 2015}
\begin{abstract}
We present a Lagrange--Galerkin scheme free from numerical quadrature for convection-diffusion problems. 
Since the scheme can be implemented exactly as it is, theoretical stability result is assured.  
While conventional Lagrange--Galerkin schemes may encounter the instability caused by numerical quadrature error, 
the  present scheme is genuinely stable.  
For the $\pk k$-element we prove error estimates 
of $O(\Delta t + h^2 + h^{k+1})$
in $\ell^\infty(L^2)$-norm and 
of $O(\Delta t + h^2 + h^k)$
in $\ell^\infty (H^1)$-norm.
Numerical results reflect these estimates.
%
\end{abstract}
\maketitle
\section{Introduction}
The Lagrange--Galerkin method, 
which is also called characteristics finite element method or Galerkin-characteristics method,
is a powerful numerical method for flow problems such as the convection-diffusion equations and the Navier--Stokes equations.
In this method the material derivative is discretized along the characteristic curve, which originates the robustness for convection-dominated problems.  
Although, as a result of the discretization along the characteristic curve, a composite function term at the previous time appears, it is converted to the right-hand side in the system of the linear equations.  
Thus, the coefficient matrix in the left-hand side is symmetric, which allows us to use efficient linear solvers for symmetric matrices such as the conjugate gradient method and the minimal residual method \cite{templates,S-2003}.
\par
Stability and error analysis of LG schemes has been done in   
\cite{AchdouGuermond2000,BermejoSaavedra2012,BMMR-1997,DouglasRussell1982,MPS,NotsuTabataOseen2015,Pironneau1982,PironneauTabata2010,Priestley1994,RuiTabata2002,Suli};  
see also the bibliography therein. 
Pironneau \cite{Pironneau1982} analyzed convection-diffusion problems and the Navier--Stokes equations to obtain suboptimal convergence results.  
Optimal convergence results were obtained by Douglas--Russell \cite{DouglasRussell1982} for convection-diffusion problems and by S\"uli \cite{Suli} for the Navier--Stokes equations.  
Optimal convergence results of second order in time were obtained by Boukir et al. \cite{BMMR-1997} for the Navier--Stokes equations in multi-step method and by Rui--Tabata \cite{RuiTabata2002} for convection-diffusion problems 
in single-step method.  
All these theoretical results are derived under the condition that the integration of the composite function term is computed exactly.  
Since, in real problems, it is difficult to get the exact integration value, numerical quadrature is usually employed.    
It is, however, reported that instability may occur caused by numerical quadrature error in \cite{MPS,Tabata2007,TabataFujima2006}.  
That is, the theoretical stability results may collapse by the introduction of numerical quadrature.   

Several methods have been studied to avoid the instability.  
The map of a particle from a time to the previous time along the trajectory, which is nothing but to solve a system of ordinary differential equations (ODEs), is simplified in \cite{BermejoSaavedra2012,MPS,Priestley1994}.  
Morton--Priestley--Suli \cite{MPS} solved the ODEs only at the centroids of the elements, and Priestley \cite{Priestley1994} did only at the vertices of the elements.  
The map of the other points is approximated by linear interpolation of those values.    
It becomes possible to perform the exact integration of the composite function term with the simplified map.  
Bermejo--Saavedra \cite{BermejoSaavedra2012} used the same simplified map as  \cite{Priestley1994} to employ a numerical quadrature of high accuracy to the composite function term.  
Tanaka--Suzuki--Tabata \cite{TSTeng} approximated the map by a locally linearized velocity and the backward Euler approximation for the solution of the ODEs in $\pk1$-element.  
The approximate map makes possible the exact integration of the composite function term with the map.  
Pironneau--Tabata \cite{PironneauTabata2010} used mass lumping in $\pk1$-element to develop a scheme free from quadrature for convection-diffusion problems.  

In this paper we prove the stability and convergence for the scheme with the same approximate map as \cite{TSTeng} in  ${P_k}$-element for convection-diffusion problems.    
Since we neither solve the ODEs nor use numerical quadrature, our scheme can be precisely implemented to realize the theoretical results.  
It is, therefore, a genuinely stable Lagrange--Galerkin scheme. 
Our convergence results are 
of $O(\Delta t + h^2 + h^{k+1})$ in $\ell^\infty(L^2)$-norm and 
of $O(\Delta t + h^2 + h^k)$ in $\ell^\infty (H^1)$-norm.  
They are best possible in both norms for $P_1$-element and in $\ell^\infty(H^1)$-norm for $P_2$-element

The contents of this paper are as follows.  
In the next section we describe the convection-diffusion problem and some preparation.
In section \ref{sec:scheme2}, after recalling the conventional Lagrange--Galerkin scheme, we present 
our genuinely stable Lagrange--Galerkin scheme. 
In section \ref{sec:mainresults} we show stability and convergence results, which are proved in section \ref{sec:proofs}.
In section \ref{sec:numer} we show some numerical results, which reflect the theoretical convergence order.  
In section \ref{sec:conclusion} we give conclusions.
\section{Preliminaries}
We state the problem and prepare notation used throughout this paper.

Let $\Omega$ be a polygonal or polyhedral domain of $\R^d ~ (d=2,3)$ and
$T>0$ be a time.
We use the Sobolev spaces $L^p(\Omega)$ with the norm $\normsob{\cdot}{0}{p}$, 
$\sobw{s}{p}(\Omega)$ and $W^{s,p}_0(\Omega)$ with the norm $\normsob{\cdot}{s}{p}$ and the semi-norm $\snormsob{\cdot}{s}{p}$ for $1\leq p \leq \infty$ and a positive integer $s$.
We will write $H^s(\Omega) = W^{s,2}(\Omega)$ and drop the subscript $p=2$ in the corresponding norms.
The $L^2$-norm $\normsob{\cdot}{0}{2}$ is simply denoted by $\mynorm{\cdot}$.
The dual space of $H^1_0(\Omega)$ is denoted by $H^{-1}(\Omega)$.
For the vector-valued function $w\in W^{1,\infty}(\Omega)^d$ we define the semi-norm $\snormsobi{w}{1}{\infty}$ 
by
\begin{equation*}
    \Biggl\|
	\biggl\{
         \sum_{i,j=1}^d \Bigl( \henbi{w_i}{x_j} \Bigr)^2
    \biggr\}^{1/2}
    \Biggr\|_{0,\infty}.
\end{equation*}
The parenthesis $(\cdot , \cdot)$ shows the $L^2$-inner product $(f,g) \equiv \int_{\Omega} f g ~ dx$.
For a Sobolev space $X(\Omega)$ 
we use abbreviations 
$H^m(X)=H^m(0,T;X(\Omega))$ and $C(X)=C([0,T];X(\Omega))$.
We define a function space $Z^m(t_1,t_2)$ by
\begin{equation*}
\begin{split}
Z^m(t_1,t_2) &\equiv \{
	f\in H^j(t_1, t_2;H^{m-j}(\Omega))
	; j=0,\dots,m , 
	\mynorm{f}_{Z^m(t_1,t_2)}<\infty
	\}
	, \\
\mynorm{f}_{Z^m(t_1,t_2)} &\equiv
\biggl\{ \sum_{j=0}^m \mynorm{f}_{H^j(t_1,t_2;H^{m-j})}^2 \biggr\}^{1/2}
\end{split}
\end{equation*}
and denote $Z^m(0,T)$ by $Z^m$. 

We consider the convection-diffusion problem:
find $\phi: \Omega \times (0,T) \to \mathbb R$
such that 
\begin{subequations}\label{convdifeq}
\begin{align}
	\henbi{\phi}{t} + u\cdot \nabla \phi - \nu \Delta \phi &= f,   \qquad (x,t) \in \Omega \times (0,T), 	\label{convdifeqa}\\
	\phi &= 0, \qquad (x,t) \in \partial \Omega \times (0,T), \label{convdifeqb}\\
	\phi &= \phi^0,  \qquad x\in \Omega, \, t = 0, 	\label{convdifeqc}
\end{align}
\end{subequations}
where 
$\partial \Omega$ is the boundary of $\Omega$ and
$\nu >0$ is a diffusion constant which is less than or equal to a given $\nu_0$. 
Functions 
$u: \Omega \times (0,T) \to \mathbb R^d$, 
$f\in C(L^2)$ and 
$\phi^0\in C(\bar \Omega)$
are given.
\begin{remark}
As usual, in place of (\ref{convdifeqb}), we can deal with the inhomogeneous boundary condition $\phi=g$ by replacing the unknown function $\phi$ by $\widetilde \phi \equiv \phi - \widetilde g$ if the function $g$ defined on $\partial \Omega \times (0,T)$ can be extended to a function $\widetilde g$ in $\Omega \times (0,T)$ appropriately.
\end{remark}

Let 
$\Delta t>0$ be a time increment, $N_T \equiv \lfloor T/\Delta t \rfloor$, 
$t^n \equiv n\Delta t$ and 
$\psi^n \equiv \psi(\cdot,t^n)$ for a function $\psi$ defined in $\Omega \times (0,T)$.
For a set of functions 
$\psi=\set{\psi^n}_{n=0}^{N_T}$, 
two norms 
$\mynorm{\cdot}_{\ell^\infty(L^2)}$ and $\mynorm{\cdot}_{\ell^2(n_1, n_2;L^2)}$
are defined by
\begin{equation*}
\begin{split}
	\norm{\psi}_{\ell^\infty(L^2)} &\equiv 
	\max\{\norm{\psi^n};n=0,\dots ,N_T\}, \\
	\norm{\psi}_{\ell^2(n_1,n_2;L^2)} &\equiv 
	\left(
	\Delta t \sum_{n=n_1}^{n_2} \norm{\psi^n}^2 
	\right)^{1/2}
\end{split}
\end{equation*}
and 
denote $\mynorm{\psi}_{\ell^2(1,N_T;L^2)}$ by $\mynorm{\psi}_{\ell^2(L^2)}$.

Let $u$ be smooth. 
The characteristic curve $X(t; x,s)$ 
is defined by the solution of the system of the ordinary differential equations, 
\begin{subequations}
\begin{align}
	\frac{dX}{dt}(t; x,s)&=u(X(t;x,s),t), \quad t<s,\\
	X(s;x,s)&=x.
\end{align}
\end{subequations}
Then, we can write the material derivative term $\henbi{\phi}{t}+u\cdot \nabla \phi$ as
\begin{equation*}
\left(\henbi{\phi}{t} +u \cdot \nabla \phi \right)(X(t),t)
=\frac{d}{dt} \phi(X(t),t).
\end{equation*}
For $w:\Omega \to \R^d$ we define the mapping $X_1(w):\Omega \to \R^d$ by
\begin{equation}\label{eq:x1def}
(X_1(w))(x) \equiv x - w(x)\Delta t.
\end{equation}
\begin{remark}
The image of $x$ by $X_1(u(\cdot,t))$ is nothing but the backward Euler approximation
of $X(t-\Delta t;x,t)$.
\end{remark}
The symbol $\circ$ stands for the composition of functions, 
e.g., $(g\circ f)(x) \equiv g(f(x))$.

Let $\mathcal T_h \equiv \{ K \} $ be a triangulation of $\bar \Omega$ and
$h \equiv \max_{K\in \mathcal{T}_h} \operatorname{diam}(K)$ be the maximum element size.  
Throughout this paper we consider a regular family of triangulations $\{ \mathcal{T}_h  \}_{h\downarrow 0}$.  
Let $k$ be a fixed positive integer and 
$V_h\subset H_0^1(\Omega)$ be the $\pk k$-finite element space, 
\[
  V_h \equiv \{  v_h \in C({\bar \Omega}) \cap H_0^1(\Omega); ~v_{h |K} \in P_k(K), ~\forall K \in   \mathcal{T}_h   \}  ,  
\]
where $\pk k (K)$ is the set of polynomials on $K$ whose degrees are less than or equal to $k$.
Let $\pro{\phi}\in V_h$ be the Poisson projection of $\phi\in H_0^1(\Omega)$ defined by  
\begin{equation}\label{eq:poissonProDef}
 (\nabla (\pro{\phi}-\phi),\nabla \psi_h)=0, \quad \forall \psi_h\in V_h.
\end{equation}

We use $c$ to represent a generic positive constant independent of $h$, $\Delta t$, $\nu$, $f$ and $\phi$ which may take different values at different places.
The notation $c(A)$ means that $c$ depends on a positive parameter $A$ and that $c$ increases monotonically when $A$ increases.
The constants $c_0$, $c_1$ and $c_2$ stand for $c_0=c(\mynorm{u}_{C(L^{\infty})})$, $c_1=c(\mynorm{u}_{C(W^{1,\infty})})$ and $c_2=c(\mynorm{u}_{C(W^{2,\infty})})$.
We also use fixed positive constants $\cuint$ and $\done$  defined in Lemma \ref{lemm:interpolation} in the next section  
and in Lemma \ref{jacobiest} in Section \ref{sec:proofs}, respectively.  
\section{
A genuinely stable Lagrange--Galerkin scheme 
}\label{sec:scheme2}

The conventional Lagrange--Galerkin scheme, which we call \schemezero,  is 
described as follows.
\begin{Thschemezero}
Let $\phi_h^0 = \pro{\phi}^0$.
Find $\set{\phi_h^n}_{n=1}^{N_T} \subset V_h$ 
such that 
for $ n=1,\dots ,N_T$ 
\begin{equation}\label{eq:scheme0}
\begin{split}
	\left(  \frac{\phi_h^{n} - \phi_h^{n-1} \circ X_{1}^{n} }{\Delta t},\psi_h \right)
	+ \nu (\nabla \phi_h^{n},\nabla \psi_h)  
	= (f^{n}, \psi_h), \quad 
	\forall \psi_h\in V_h, 
\end{split}
\end{equation}
where $X_1^{n}=X_1(u^{n})$.
\end{Thschemezero}
For this scheme error estimates 
\begin{equation}\label{eq:estimateScheme1}
\begin{split}
	\mynorm{\phi_h-\phi}_{\ell^\infty(L^2)} &\leq c(h^k+\Delta t), ~ c(1/\nu)(h^{k+1}+\Delta t),\\
	\mynorm{\phi_h-\phi}_{\ell^\infty(H^1)} &\leq c(1/\nu)(h^k+\Delta t)
\end{split}
\end{equation}
are proved in \cite{DouglasRussell1982}, 
where the composite function term
	$( \phi_h^{n-1} \circ X_1^{n}, \psi_h )$
is assumed to be exactly integrated.

Although the function $\phi_h^{n-1}$ is a polynomial on each element $K$, the composite function $\phi_h^{n-1} \circ X_1^{n}$ is not a polynomial on $K$ in general 
since the image $X_1^{n}(K)$ of an element $K$ may spread over plural elements.
Hence, it is hard to calculate the composite function term $( \phi_h^{n-1} \circ X_1^{n}, \psi_h )$ exactly.
In practice, the following numerical quadrature has been used.
Let $g:K \to \R$ be a continuous function.  
A numerical quadrature $I_h[g;K]$ of $\int_K g \, dx$ is defined by 
\begin{equation}\label{eq:numericalQuadrature}
  I_h[g;K] \equiv \meas{K} \sum_{i=1}^{N_q} w_i \ g(a_i), 
\end{equation}
where $N_q$ is the number of quadrature points and $(w_i, a_i)\in \R \times K$ is a pair of weight and point for $i=1,\dots, N_q$.
We call the practical scheme using numerical quadrature \jurai.   
\begin{Thschemejurai}
Let $\phi_h^0 = \pro{\phi}^0$.
Find $\set{\phi_h^n}_{n=1}^{N_T} \subset V_h$ 
such that 
for $ n=1,\dots ,N_T$ 
\begin{equation}\label{eq:scheme1}
\begin{split}
	\frac{1}{\Delta t}(\phi_h^{n},\psi_h) 
	- \frac{1}{\Delta t} \sum_{K \in \mathcal T_h} I_h[(\phi_h^{n-1} \circ X_{1}^{n})\psi_h;K]
	& + \nu (\nabla \phi_h^{n},\nabla \psi_h)  
	\\ = (f^{n}, \psi_h),
	 & \qquad \forall \psi_h\in V_h,  
\end{split}
\end{equation}
where $X_1^{n}=X_1(u^{n})$.
\end{Thschemejurai}
It is reported that the results (\ref{eq:estimateScheme1}) do not hold for \jurai \ \cite{MPS,Tabata2007,TabataFujima2006,TSTeng}.

We denote by $\Pi_h^{(1)}$ the Lagrange interpolation operator to the $\pk 1$-finite element space. 
The following lemma is well-known \cite{Ciarlet}.   
\begin{lemma}\label{lemm:interpolation}
\begin{enumerate}
	\item 
	There exists a positive constant $\cint$ such that for
	$w \in W^{2,\infty}(\Omega)^d$  
	\begin{equation*}
		\normsobi{\Pi_h^{(1)} w -w}{0}{\infty} \leq \cint  h^{2} \snormsob{w}{2}{\infty}.
	\end{equation*}
	\item There exists a positive constant $\cuint \geq 1$ such that for $w\in W^{1,\infty}(\Omega)^d$ 
	\begin{equation*}
		\snormsobi{\Pi_h^{(1)} w}{1}{\infty} \leq \cuint \snormsob{w}{1}{\infty}.
	\end{equation*}
\end{enumerate}
\end{lemma}

We now present our genuinely stable scheme GSLG, which is free from quadrature and exactly computable.  
We define a locally linearized velocity $u_h$ and a mapping $X_{1h}^n$ by
\begin{equation*}
	u_h \equiv \Pi_h^{(1)} u, \quad 
	X_{1h}^n \equiv X_1(u_h^n).
\end{equation*}
\begin{Thschemezettai}
  Let $\phi_h^0 = \pro{\phi}^0$.
  Find $\set{\phi_h^n}_{n=1}^{N_T} \subset V_h$ 
  such that 
  for $n=1,\dots ,N_T$ 
  \begin{equation}\label{eq:scheme2}
  \begin{split}
	\left(  \frac{\phi_h^{n} - \phi_h^{n-1} \circ X_{1h}^{n} }{\Delta t},\psi_h \right)
	+ \nu (\nabla \phi_h^{n},\nabla \psi_h) 
	= (f^{n}, \psi_h), \quad
	\forall \psi_h\in V_h. 
  \end{split}
  \end{equation}
\end{Thschemezettai}
%
%
We 
show that the integration
 $(\phi_h^{n-1}\circ X_{1h}^{n},\psi_h)$
can be calculated exactly.
\onefig{
	\includegraphics[height=30mm]{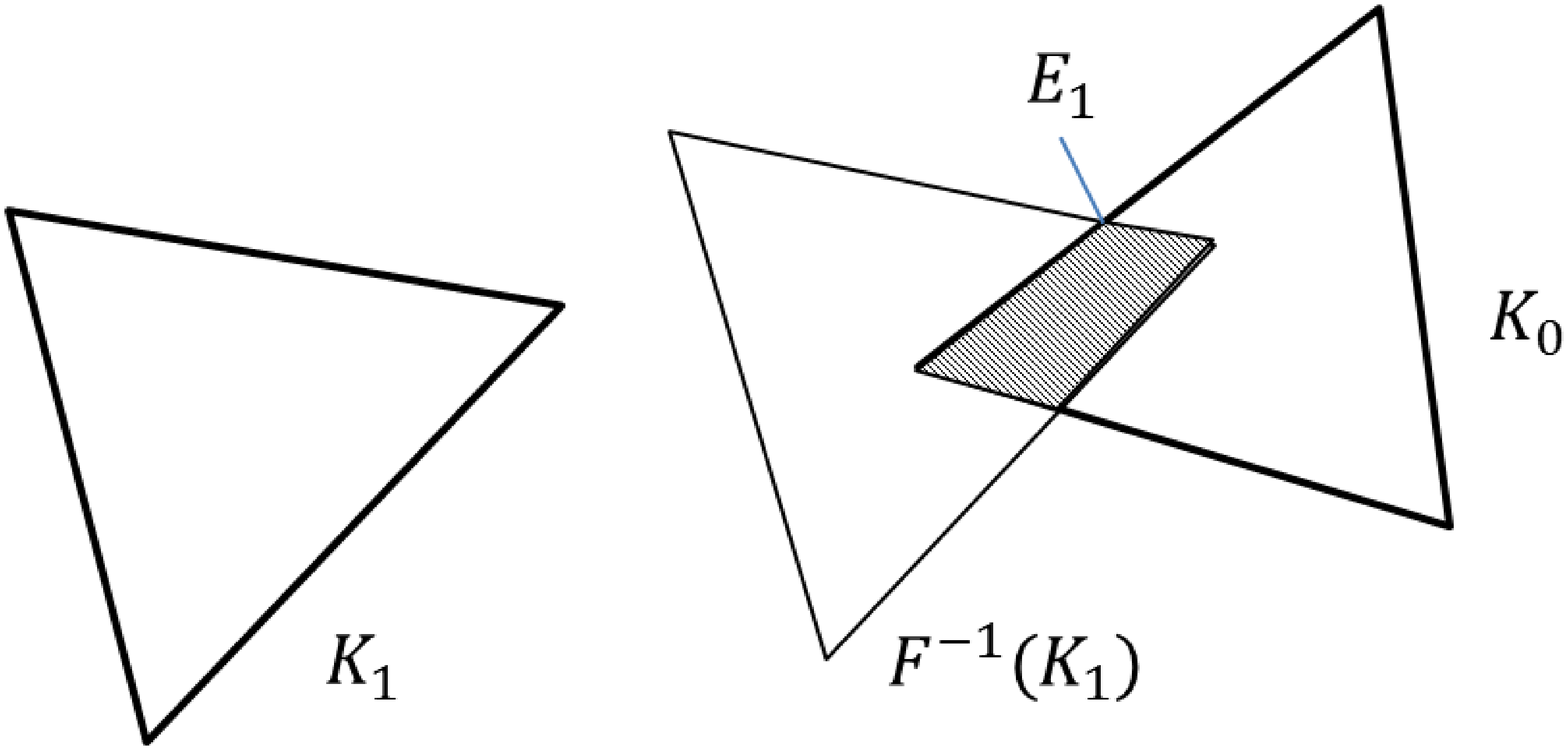}
	\caption{Elements $K_0$, $K_1$ and a polygon $E_1$}
	\label{fig:el}
}
At first we prepare two lemmas.  
The next lemma on the mapping \eqref{eq:x1def} is proved in {\cite{RuiTabata2002}}.  
\begin{lemma}[{\cite[Proposition 1]{RuiTabata2002}}]\label{lemm:bijective}
  Suppose 
  \begin{equation}
  w\in W_0^{1,\infty}(\Omega)^d \text{ and } \Delta t  \snormsob{w}{1}{\infty} < 1.
  \end{equation}
  Let 
  $F\equiv X_1(w)$ be the mapping defined in (\ref{eq:x1def}). 
  Then, 
  $F :\Omega \to \Omega$ is bijective.
\end{lemma}
\begin{lemma}\label{lemm:polygon}
 Let $K_0, K_1\in \mathcal T_h$ and $F:K_0\to \R^d$ be linear and one-to-one.
 Let 
 $
	E_1 \equiv K_0 \cap F^{-1} (K_1)
 $
 and $\meas{E_1}>0$. 
 Then, the following hold.
 \begin{enumerate}
 \item $E_1$ is a polygon ($d=2$) or a polyhedron ($d=3$).
 \item $\phi_h \circ F_{|E_1}\in \pk k (E_1), \quad \forall \phi_h \in \pk k(K_1)$.
 \end{enumerate}
\end{lemma}
\begin{proof}
(i) Since both $K_0$ and $F^{-1}(K_1)$ are triangles ($d=2$) or tetrahedra ($d=3$), the intersection is a polygon or a polyhedron.
See Fig. \ref{fig:el}.
\par
(ii) $F\in \pk 1 (K_0)^{d}$ implies that $F\in \pk 1(E_1)^d$
and it holds that 
$F(E_1)\subset K_1$.  Hence, $\phi_h \circ F_{|E_1}$ is well defined and
$\phi_h \circ F_{|E_1} \in \pk k (E_1)$.
\end{proof}
\begin{proposition}
  Let $\phi_h$, $\psi_h\in V_h$, $w\in W^{1,\infty}_0(\Omega)$ and 
  $X_{1h}\equiv X_1(\Pi_h^{(1)} w)$, where $X_1$ is the operator defined in (\ref{eq:x1def}).
  Suppose 
  $ \cuint \Delta t \snormsob{w}{1}{\infty} <1$.
  Then, $\int_{\Omega} (\phi_h \circ X_{1h}) \psi_h dx$ is exactly computable.
\end{proposition}
\begin{proof}
It is sufficient to show that $\int_{K_0} (\phi_h \circ X_{1h}) \psi_h dx$ can be computable exactly for any $K_0 \in \mathcal T_h$.
The mapping $X_{1h}:\Omega \to \Omega$ is bijective 
since we can apply Lemma \ref{lemm:bijective} thanks to
\begin{equation}\label{eq:dtModified}
	\Delta t  \snormsobi{\Pi_h^{(1)} w}{1}{\infty}
\le 
	\cuint	\Delta t  \snormsobi{w}{1}{\infty}
	< 1 .
\end{equation} 
Let $\Lambda(K_0) \equiv \set{l; K_0\cap X_{1h}^{-1}(K_l) \not = \emptyset}$ and $E_l \equiv K_0 \cap X_{1h}^{-1}(K_l)$ for $l\in \Lambda(K_0)$.
Noting that 
\begin{equation*}
	\bigcup_{l\in \Lambda(K_0)} E_l
	=K_0 \cap \bigcup_{l\in \Lambda(K_0)} X_{1h}^{-1}(K_l)
	= K_0
\end{equation*}
and that $\meas{E_l\cap E_m}=0$ for $l \not=m$,
$l,m \in \Lambda (K_0)$, 
we can divide 
the integration on $K_0$ into the sum of those on $E_l$ for $l \in \Lambda(K_0)$, 
\begin{equation*}
	\int_{K_0} (\phi_h\circ X_{1h}) \psi_h dx
	= \sum_{l\in \Lambda(K_0)} \int_{E_l} (\phi_h\circ X_{1h}) \psi_h dx.
\end{equation*}
Since Lemma \ref{lemm:polygon} with $F=X_{1h}$ implies that
both $\phi_h\circ X_{1h}$ and $\psi_h$ are polynomials on $E_l$, we can execute the exact integration.
\end{proof}
\begin{remark}
In the case of $d=2$, 
Priestley \cite{Priestley1994} approximated
$X(t^{n-1};x,t^n)$ 
by
\begin{equation*}
	\widetilde X_{1h}(x) = B_1 \lambda_1(x)+B_2 \lambda_2(x)+B_3 \lambda_3(x), ~ x\in K_0
\end{equation*}
on each $K_0 \in \mathcal T_h$, where $B_i=X(t^{n-1};A_i,t^n)$, $\set{A_i}_{i=1}^3$ are vertices of $K_0$ and $\set{\lambda_i}_{i=1}^3$ are the barycentric coordinates of $K_0$ with respect to $\set{A_i}_{i=1}^3$.
Since $\widetilde X_{1h}(x)$ is linear in $K_0$, the decomposition
\begin{align*}
	\int_{K_0} (\phi_h \circ \widetilde X_{1h}) \psi_h dx
	= \sum_{l \in \Lambda(K_0)} \int_{E_l} (\phi_h \circ \widetilde X_{1h}) \psi_h dx, \\
	\Lambda(K_0) \equiv \set{l; K_0\cap \widetilde X_{1h}^{-1}(K_l) \not = \emptyset},
	E_l \equiv K_0 \cap \widetilde X_{1h}^{-1}(K_l) 
\end{align*}
makes the exact integration possible. 
However, $B_i = X(t^{n-1};A_i,t^n)$ are the solutions of a system of ordinary differential equations and
they cannot be solved exactly in general.
In practice, some numerical method, e.g., Runge--Kutta method, is required, which introduces another error.
\end{remark}
%
%
%
\section{Main results}\label{sec:mainresults}
We show the main results, the stability and convergence of \zettai.
\begin{hypo} \label{hypo:uonly}
	(i) 	$u\in C((W_0^{1,\infty})^d)$,  
	(ii) 
	$u\in C((W_0^{1,\infty} \cap W^{2,\infty})^d)$.
\end{hypo}
\begin{hypo} \label{hypo:phi}
	$\phi \in H^1(H^{k+1}) \cap Z^2$.
\end{hypo}
\begin{hypo}\label{hypo:hypothesis1}
The time increment $\Delta t$ satisfies
$0<\Delta t\leq \dtgiven$,
where 
\begin{equation}
\dtgiven \equiv \frac{\done}{\cuint \matrixnorm{u}},
\end{equation}
and
$\cuint$ and $\done$ are the constants stated in 
Lemma \ref{lemm:interpolation} (Section \ref{sec:scheme2}) and Lemma \ref{jacobiest} (Section \ref{sec:proofs}), 
respectively.
\end{hypo}
\begin{hypo}\label{hypo:poissonProjection}
	There exists a positive constant $\cpro$ such that, for $\psi\in H^{k+1}(\Omega) \cap H_0^1(\Omega)$, 
	\begin{equation}
		\normsobi{\pro{\psi}-\psi}{0}{2} \leq \cpro h^{k+1} \normsob{\psi}{k+1}{2}, 
	\label{h2regularity}
	\end{equation}
	where $\pro{\psi}$ is the Poisson projection defined in (\ref{eq:poissonProDef}).
\end{hypo}
\begin{remark}
(i) It is well-known that the $H^1$-estimate 
	\begin{equation}
		\normsobi{\pro{\psi}-\psi}{1}{2} \leq \cpro h^{k} \normsob{\psi}{k+1}{2}
         \label{h1estimate}
		\end{equation}
	holds without any specific condition.  
    On the other hand, Hypothesis \ref{hypo:poissonProjection} holds, for example, if $\Omega$ is convex, by Aubin--Nitsche lemma \cite{Ciarlet}. \\
(ii) Hypothesis \ref{hypo:phi} implies $\phi \in C(H^{k+1})$ and $\phi^0\in H^{k+1}(\Omega)$.
\end{remark}
\begin{theorem}\label{theo:mainStability}
Suppose Hypotheses
\ref{hypo:uonly}-(i) 
and \ref{hypo:hypothesis1}.  
Then,
there exists a positive constant 
$\cprop$ independent of $h, \Delta t, \nu, \phi$ and $f$
such that 
\begin{equation*}
	\mynorm{\phi_h}_{\ell^\infty (L^2)}
	+ \sqrt{\nu} \mynorm{\nabla \phi_h}_{\ell^2(L^2)}
	\leq \cprop \left( \mynorm{\phi_h^0}+\mynorm{f}_{\ell^2(L^2)} \right).
\end{equation*}
\end{theorem}
\begin{theorem}\label{theo:mainConvergence}
Suppose Hypotheses \ref{hypo:uonly}-(ii), \ref{hypo:phi} and \ref{hypo:hypothesis1}. \\
(i)
There exists a positive constant 
$\cprop$ independent of $h, \Delta t, \nu$ and $\phi$
such that 
\begin{equation}\label{eq:mainest1}
\begin{split}
	& \norm{\phi-\phi_h}_{\ell^\infty(L^2)}
	+ \sqrt{\nu}\norm{\nabla(\phi - \phi_h)}_{\ell^2(L^2)} \\ 
	\leq & \cprop \biggl\{ 
	\Delta t \mynorm{ \phi }_{Z^2}  
	+ h^{k} 
	\Bigl(  
		\norm{ \henbi{\phi}{t} }_{L^2(H^{k+1})} 
		  +\mynorm{\phi}_{\ell^\infty(H^{k+1})}
		 + \mynorm{\phi}_{\ell^2(0,N_T;H^{k+1})}
	\Bigr) \\
	&+ h^2 \mynorm{\nabla \phi}_{\ell^2(0,N_T-1;L^2)}
	\biggr\}.
\end{split}\end{equation}
(ii)
There exists a positive constant $\cproptwo$ independent of $h,\Delta t,\phi$ (but dependent on $1/\nu$) such that 
\begin{equation}\label{eq:mainest2}
\begin{split}
	&\norm{\phi-\phi_h}_{\ell^\infty(H^1)}
	\leq \cproptwo
	\biggl\{ 
	\Delta t \mynorm{ \phi }_{Z^2}
	+ h^{k} 
	\biggl(  
	\norm{ \henbi{\phi}{t} }_{L^2(H^{k+1})}  
	+\mynorm{\phi}_{\ell^\infty(H^{k+1})} \\
	&+ \mynorm{\phi}_{\ell^2(0,N_T-1;H^{k+1})}
	\biggr) 
	+ h^2 \mynorm{\nabla \phi}_{\ell^2(0,N_T-1;L^2)}
	\biggr\}.
\end{split}
\end{equation}
(iii)
Moreover, suppose 
Hypothesis \ref{hypo:poissonProjection}.
Then, there exists a positive constant $\cpropthree$ independent of $h, \Delta t, \nu$ and $\phi$ such that 
\begin{equation}\label{eq:mainest3}
\begin{split}
	\norm{\phi-\phi_h}_{\ell^\infty(L^2)}
	& \leq \cpropthree 
	\biggl\{
	 \Delta t \mynorm{ \phi }_{Z^2} 
	+ h^{k+1} \Bigl(  
		\norm{ \henbi{\phi}{t} }_{L^2(H^{k+1})} 
		 +\mynorm{\phi}_{\ell^\infty(H^{k+1})} \\
		&+ \nu^{-1/2} \mynorm{\phi}_{\ell^2(0,N_T-1;H^{k+1})}
	\Bigr) 
	+ h^2 \mynorm{\nabla \phi}_{\ell^2(0,N_T-1;L^2)}
	\biggr\}.
\end{split}\end{equation}
\end{theorem}
\begin{remark}
From Theorem \ref{theo:mainConvergence}, 
we have 
\begin{equation*}
\begin{split}
	\mynorm{\phi-\phi_h}_{\ell^\infty(L^2)}&\leq c(\Delta t + h^2 + h^k ), \quad
	c\left(\Delta t + h^2 +  \frac{1}{\sqrt{\nu}}h^{k+1}\right)\\
	\mynorm{\phi-\phi_h}_{\ell^\infty (H^1)}&\leq c\left(\frac{1}{\nu}\right)(\Delta t + h^2 + h^k ).
\end{split}
\end{equation*}
	In the case of $\pk k$-element, $k=1,2$, the estimate (\ref{eq:mainest1}) shows the optimal 
	$L^2$-convergence rate $O(\Delta t + h^k)$ independent of $\nu$.
	The dependency on $\nu$ in (\ref{eq:mainest2}) and (\ref{eq:mainest3}) is also inevitable in \schemezero.
\end{remark}
\section{Proofs of main theorems}\label{sec:proofs}
We recall some results used in proving main theorems. 
For their proofs we only show outlines or refer to the bibliography.  
\begin{lemma}[{\cite[Lemma 1]{RuiTabata2002}}] \label{lemm:bijective_1}
  Suppose 
  $w\in W_0^{1,\infty}(\Omega)^d $ and  
  \begin{equation}
  \Delta t  \snormsob{w}{1}{\infty} <1.
  \end{equation}
  Let 
  $F\equiv X_1(w)$ be the mapping defined in (\ref{eq:x1def}). 
  Then, 
  there exists a positive constant $c(\snormsob{w}{1}{\infty})$ such that for $\psi\in L^2(\Omega)$ 
  \begin{equation*}
  	\mynorm{\psi \circ F} 
	\leq (1+ c \Delta t) \mynorm{\psi}.
  \end{equation*}
\end{lemma}
The proof is given in \cite{RuiTabata2002}.  
\begin{lemma}\label{jacobiest}
	There exists a constant $\done \in (0,1)$ such that, for $w\in W^{1,\infty}_0(\Omega)^d$ and 
	$\Delta t$ satisfying $\Delta t\abs{w}_{1,\infty} \leq \done$,
	\begin{equation*}
		\frac12 \leq \jac{X_1(w)}{x} \leq \frac32, 
	\end{equation*}
	where $|\partial X_1(w)/\partial x|$ is the Jacobian of the mapping $X_1(w)$ defined in (\ref{eq:x1def}).
\end{lemma}
Lemma \ref{jacobiest} is easily proved by the fact, 
\[   \left( \frac{ \partial X_1(w)}{\partial x} \right)_{ij}=\delta_{ij}-\Delta t ~\frac{\partial w_i}{\partial x_j}.  
\]
\begin{lemma} \label{lemm:twofunc}
Let $w_i\in W_0^{1,\infty}(\Omega)^d$ and $F_i\equiv X_1(w_i)$ be the mapping defined in (\ref{eq:x1def}) for $i=1,2$.
Under the condition  
$\Delta t \snormsob{w_i}{1}{\infty} \le \done$, $i=1,2$, 
we have 
for $\psi\in H^1(\Omega)$  
\begin{equation*}
	\mynorm{\psi\circ F_1 - \psi\circ F_2} \leq \sqrt{2} 
	\Delta t\normsob{w_1-w_2}{0}{\infty} 
	\mynorm{\nabla \psi}.
\end{equation*} 
\end{lemma}
Lemma \ref{lemm:twofunc} is a direct consequence of \cite[Lemma 4.5]{AchdouGuermond2000} and Lemma \ref{jacobiest}.  
\begin{lemma}
\label{lemm:DouglasRussell}
Let $w\in W_0^{1,\infty}(\Omega)^d$ and $F\equiv X_1(w)$ be the mapping defined in (\ref{eq:x1def}). 
Under the condition $\Delta t \snormsob{w}{1}{\infty} \leq \done $,
there exists a positive constant 
$
c(\normsob{w}{1}{\infty})$
such that for $\psi \in L^2(\Omega)$
\begin{equation*}
\mynorm{\psi-\psi\circ F}_{H^{-1}(\Omega)} \leq 
c
\Delta t  \mynorm{\psi}.
\end{equation*}
\end{lemma}
Lemma \ref{lemm:DouglasRussell} is obtained from \cite[Lemma 1]{DouglasRussell1982} and Lemma \ref{jacobiest}.  
\begin{lemma}[discrete Gronwall inequality] 
\label{lem:discreteGronwall}
Let $a_0$ and $a_1$ be non-negative numbers, $\Delta t \in (0,\frac{1}{2a_0}]$ be a real number, and $\set{x^n}_{n\geq 0}, \set{y^n}_{n\geq 1}$ and $\set{b^n}_{n\geq 1}$ be non-negative sequences.
Suppose
\begin{equation*}
	\frac{x^n-x^{n-1}}{\Delta t} + y^n \leq a_0 x^n + a_1 x^{n-1} + b^n, ~ \forall n \geq 1.
\end{equation*}
Then, it holds that 
\begin{equation*}
	x^n + \Delta t \sum_{i=1}^n y^i \leq \exp\set{(2a_0+a_1)n\Delta t} \left( x^0 + \Delta t \sum_{i=1}^n b^i \right), 
	~ \forall n \geq 1.
\end{equation*}
\end{lemma}
Lemma \ref{lem:discreteGronwall} is shown by using the inequalities
\begin{equation*}
	\frac{1}{1-a_0\Delta t} \leq 1+2a_0\Delta t \leq \exp(2a_0\Delta t).
\end{equation*} 
{\it Outline of the proof of Theorem \ref{theo:mainStability}}.  
We substitute $\phi_h^n$ into $\psi_h$ in \eqref{eq:scheme2}.   
We can apply Lemma \ref{lemm:bijective_1} with $w=u_h^n$ and $\psi=\phi_h^{n-1}$ by virtue of $\Delta t \matrixnorm{u_h} <1$. 
The rest of the proof is similar to \cite[Theorem 1]{RuiTabata2002}.  
We, therefore, omit it.  
\par
\renewcommand{\proofname}{Proof of Theorem \ref{theo:mainConvergence}}
\begin{proof}
We first show the estimate (\ref{eq:mainest1}).
Let 
\begin{equation}\label{eq:etilde}
	\eone \equiv \phi_h-\pro{\phi}, ~ \proe \equiv \phi-\pro{\phi},
\end{equation}
where $\pro{\phi}$ is the Poisson projection defined in (\ref{eq:poissonProDef}).
From (\ref{convdifeq}) and (\ref{eq:scheme2}) we have
\begin{equation}\label{eq:errorEquation}
	\left(\frac{\eone^n-\eone^{n-1}\circ X_{1h}^{n}}{\Delta t},\psi_h \right)
	+ \nu (\nabla \eone^n, \nabla \psi_h)
	= \sum_{i=1}^4 (R_i^n,\psi_h)
\end{equation}
for $\psi_h \in V_h$, 
where
\begin{equation}\label{eq:defR}
\begin{split}
	R_1^n &\equiv \henbi{\phi^n}{t}+u^n\cdot \nabla \phi^n 
		- \frac{\phi^n-\phi^{n-1}\circ X_1^{n}}{\Delta t},\\
	R_2^n &\equiv \frac{\phi^{n-1}\circ X_{1h}^{n}-\phi^{n-1}\circ X_{1}^{n}}{\Delta t},\\
	R_3^n &\equiv \frac{\proe^n-\proe^{n-1}}{\Delta t}, \quad
	R_4^n \equiv \frac{\proe^{n-1}-\proe^{n-1}\circ X_{1h}^{n}}{\Delta t}.
\end{split}
\end{equation}
Substituting $\eone^n$ into $\psi_h$, applying Lemma \ref{lemm:bijective_1} 
with $F=X_{1h}^{n}$ and $\psi=\eone^{n-1}$,
and evaluating the first term of the left-hand side as
\begin{align*}
	\left(\frac{\eone^n-\eone^{n-1}\circ X_{1h}^{n}}{\Delta t},e_h^n \right)
	&\geq \frac{1}{2\Delta t} (\mynorm{e_h^n}^2-\mynorm{e_h^{n-1}\circ X_{1h}^n}^2) \\
	&\geq \frac{1}{2\Delta t} (\mynorm{e_h^n}^2-(1+c_1\Delta t)^2 \mynorm{e_h^{n-1}}^2) \\
	&= \frac{1}{2\Delta t}(\mynorm{e_h^n}^2 - \mynorm{e_h^{n-1}}^2) - \frac{c_1}{2} (2+c_1\Delta t) \mynorm{e_h^{n-1}}^2,
\end{align*}
we have
\begin{equation}\label{eq:modifiedInterpole1}
\begin{split}
	&\frac{1}{2\Delta t}(\mynorm{\eone^n}^2-\mynorm{\eone^{n-1}}^2)
	+ \nu \mynorm{\nabla \eone^n}^2 \\
	\leq & 
	\cuone \mynorm{\eone^{n-1}}^2
	+\sum_{i=1}^4 \frac{1}{4\eps_i} \mynorm{R_i^n}^2
	+\left(\sum_{i=1}^4 \eps_i \right) \mynorm{\eone^n}^2,
\end{split}
\end{equation}
where $\set{\eps_i}_{i=1}^4$ are positive constants satisfying $\dtgiven \leq \frac{1}{4\eps_0}$, $\eps_0 \equiv \sum_{i=1}^4 \eps_i$.

We evaluate $R_i$, $i=1,\cdots,4$.
Setting 
\[  y(x,s)=x+(s-1)\Delta t ~ u^{n}(x), \quad t(s)=t^{n-1}+s \Delta t, 
\]
we have 
\[
	\frac{\phi^n-\phi^{n-1}\circ X_1^{n}}{\Delta t} = \frac{1}{\Delta t}\big[ \phi(y(\cdot,s),t(s)) \big]_{s=0}^1, 
\]
which implies 
\begin{align*}
	R_1^n &= \henbi{\phi^n}{t} + u^n\cdot \nabla \phi^n
		- \int_{0}^{1} \left\{ u^{n}(\cdot)\cdot \nabla \phi + \frac{\partial\phi}{\partial t} \right\} (y(\cdot,s),t(s))    ds 
\\ &
=
 \Delta t \int_{0}^{1} ds \int_{s}^{1} \left\{ \left( u^{n}(\cdot)\cdot \nabla  + \frac{\partial}{\partial t} \right)^2\phi  \right\} (y(\cdot,s_1),t(s_1)) d s_1 
\\ &
= \Delta t  \int_{0}^{1} s_1 \left\{ \left( u^{n}(\cdot)\cdot \nabla  + \frac{\partial}{\partial t} \right)^2\phi  \right\}  (y(\cdot,s_1),t(s_1))    d s_1 . 
\end{align*}
Hence, we have 
\begin{align}
\| R_1^n \| 
& \le \Delta t  \int_{0}^{1} s_1 
  \left\| 
       \left\{ \left( u^{n}(\cdot)\cdot \nabla  + \frac{\partial}{\partial t} \right)^2\phi  \right\}  (y(\cdot,s_1),t(s_1)) 
  \right\|   d s_1  \nonumber
\\
& \le c_0 \sqrt{\Delta t}
			\mynorm{\phi}_{Z^2(t^{n-1},t^n)}, 
 \label{est:r1}
\end{align}
where we have used the transformation of independent variables from $x$ to $y$ and $s_1$ to $t$, and 
the estimate $| \partial x / \partial y | \le 2$ by virtue of Lemma \ref{jacobiest}.  
\par
From $\Delta t \matrixnorm{u}, \Delta t \matrixnorm{u_h} \leq \done$,
and Lemmas \ref{lemm:twofunc} and \ref{lemm:interpolation}
it holds that
\begin{equation} 
\label{eq:estimateR2}
  \mynorm{R_2^n} \leq \sqrt{2}\normi{\nabla \phi^{n-1}} \normsobi{\Pi_h^{(1)} u^{n} - u^{n}}{0}{\infty} \leq \cutwo h^2\mynorm{\nabla \phi^{n-1}}.
\end{equation}
$R_3^n$ is evaluated as 
\begin{align}
\| R_3^n \| 
& = 
  \left\| 
	\int_{0}^{1} 
            \frac{\partial \eta}{\partial t}    (\cdot,t(s))    d s 
 \right\|
  \le  \frac{\cpro h^{k}}{\sqrt{\Delta t}} \mynorm{\henbi{\phi}{t}}_{L^2(t^{n-1},t^n;H^{k+1})}, 
  \label{est:r3}
\end{align}
where we have used (\ref{h1estimate}). 

From $\Delta t \matrixnorm{u_h} \leq \done$
and Lemma \ref{lemm:twofunc} 
it holds that 
\begin{equation}\label{estimateR4}
  \mynorm{R_4^n} \leq \sqrt{2}\normi{\nabla \proe^{n-1}} \normsobi{\Pi_h^{(1)} u^{n}}{0}{\infty} \leq \cuzero h^k \normsob{\phi^{n-1}}{k+1}{2}.
\end{equation}
Combining 
(\ref{eq:modifiedInterpole1})--(\ref{estimateR4}),
we have
\begin{equation*} 
\begin{split}
	&\frac{1}{2\Delta t} \left( \mynorm{\eone^n}^2-\mynorm{\eone^{n-1}}^2 \right)
	+ \nu \mynorm{\nabla \eone^n}^2 
	\leq  \eps_0 \mynorm{\eone^n}^2 + \cuone \mynorm{\eone^{n-1}}^2  \\
	&+\cutwo \biggl\{ 
		\Delta t \mynorm{\phi}_{Z^2(t^{n-1},t^n)}^2
		+ h^4 \mynorm{\nabla \phi^{n-1}}^2 \\
		& + \frac{h^{2k}}{\Delta t} \mynorm{\henbi{\phi}{t}}_{L^2(t^{n-1},t^n;H^{k+1})}^2 
		+ h^{2k} \normsob{\phi^{n-1}}{k+1}{2}^2
	\biggr\}.
\end{split}
\end{equation*}
From Lemma \ref{lem:discreteGronwall} we obtain for $n=1,\dots ,N_T$
\begin{equation*}
\begin{split}
	&\mynorm{e_h^n}^2 + 2\nu \Delta t \sum_{j=1}^{N_T} \mynorm{\nabla e_h^j}^2 
	\leq \cutwo \biggl( 
		\mynorm{e_h^0}^2 
		+ \Delta t^2 \mynorm{\phi}_{Z^2}^2 \\
		& + h^{2k} \mynorm{\henbi{\phi}{t}}_{L^2(H^{k+1})}^2 
		+ h^{2k} \Delta t \sum_{j=0}^{N_T-1} \normsob{\phi^j}{k+1}{2}^2	
		+ h^4 \Delta t \sum_{j=0}^{N_T-1}\mynorm{\nabla \phi^j}^2
	\biggr),
\end{split}
\end{equation*}
which implies (\ref{eq:mainest1}) by virtue of 
$e_h^0 = 0$
and the triangle inequalities, 
\begin{align}
 \mynorm{\phi-\phi_h}_{\ell^\infty(L^2)} & \le \mynorm{\eone}_{\ell^\infty(L^2)}+\mynorm{\proe}_{\ell^\infty(L^2)} \nonumber \\
  &\le \mynorm{\eone}_{\ell^\infty(L^2)}+\cpro h^k \mynorm{\phi}_{\ell^\infty(H^{k+1})}, \label{eq:triangle}
 \\
 \mynorm{ \nabla(\phi-\phi_h) }_{\ell^2(L^2)} & \le \mynorm{\nabla\eone}_{\ell^2(L^2)}+\mynorm{\nabla\proe}_{\ell^2(L^2)} \nonumber \\
  &\le \mynorm{\nabla\eone}_{\ell^2(L^2)}+\cpro h^k \mynorm{\phi}_{\ell^2(H^{k+1})}.   
  \nonumber
 \end{align}

We show the estimate (\ref{eq:mainest2}).
The equation (\ref{eq:errorEquation}) can be rewritten as 
\begin{equation*}
	\frac{1}{\Delta t}(\eone^n-\eone^{n-1},\psi_h)+\nu(\nabla \eone^n,\nabla \psi_h)= \sum_{i=1}^5 (R_i^n,\psi_h), 
\end{equation*}
where 
\begin{equation*}
	R_5^n \equiv \frac{1}{\Delta t}(\eone^{n-1}\circ X_{1h}^{n}-\eone^{n-1}).
\end{equation*}
From Lemma \ref{lemm:twofunc}  
it holds that
\begin{equation*}
	\mynorm{R_5^n} \leq \sqrt{2} \normi{\nabla \eone^{n-1}}  \normsobi{\Pi_h^{(1)}u^{n}}{0}{\infty}
	\leq \cuzero \normi{\nabla \eone^{n-1}}.
\end{equation*}
Substituting $\bwd \eone^n \equiv \frac{1}{\Delta t}(\eone^n-\eone^{n-1})$ into $\psi_h$, 
and using (\ref{est:r1})--(\ref{estimateR4}) for $R_1,\dots,R_4$,
we have
\begin{equation*}
\begin{split}
	&\mynorm{\bwd \eone^n}^2 
	+ \frac{1}{\Delta t}\left( 
	\frac{\nu}{2} \mynorm{\nabla\eone^n}^2
	- \frac{\nu}{2} \mynorm{\nabla\eone^{n-1}}^2
	\right)
	+ \frac{\nu}{2 \Delta t}\mynorm{\nabla(\eone^n-\eone^{n-1})}^2 \\
	\leq & \cutwo \biggl\{ 
	\Delta t \mynorm{\phi}_{Z^2(t^{n-1},t^n)}^2
	+ h^4 \mynorm{\nabla \phi^{n-1}}^2 
	 + \frac{h^{2k}}{\Delta t} \mynorm{\henbi{\phi}{t}}_{L^2(t^{n-1},t^n;H^{k+1})}^2 \\
	& 
	+ h^{2k} \normsob{\phi^{n-1}}{k+1}{2}^2 
	\biggr\}
	+\frac{\cuzero}{\nu} \left( \frac{\nu}{2} \mynorm{\nabla \eone^{n-1}}^2 \right) 
	+\frac{1}{2} \mynorm{\bwd \eone^{n}}^2.
\end{split}
\end{equation*}
From 
Lemma \ref{lem:discreteGronwall} we have for $n=1,\dots ,N_T$
\begin{equation*}
\begin{split}
	& \frac{\Delta t}{2} \sum_{j=1}^{N_T} \mynorm{\bwd \eone^n}^2 +
	\frac{\nu}{2}\mynorm{\nabla \eone^n}^2
	\leq \cutwo \exp\left(\frac{\cuzero T}{\nu}\right) \biggl( 
	\mynorm{\nabla \eone^0}^2 
	+ \Delta t^2 \mynorm{\phi}_{Z^2}^2 
	\\
	& 
	+ h^{2k} 
	\mynorm{\henbi{\phi}{t}}_{L^2(H^{k+1})}^2 
	+ 
	h^{2k} \Delta t \sum_{j=0}^{N_T-1} \normsob{\phi^j}{k+1}{2}^2
	+ h^4 \Delta t \sum_{j=0}^{N_T-1}\mynorm{\nabla \phi^j}^2
		\biggr),
\end{split}
\end{equation*}
which implies (\ref{eq:mainest2}) by virtue of 
$e_h^0=0$,
the triangle inequality,
\begin{equation*}
\begin{split}
	\mynorm{\nabla (\phi-\phi_h)}_{\ell^\infty(L^2)}
	&\leq \mynorm{\nabla e_h}_{\ell^\infty(L^2)}+\mynorm{\nabla \eta}_{\ell^\infty(L^2)} \\
	&\leq  \mynorm{\nabla e_h}_{\ell^\infty(L^2)} + \cpro h^k\mynorm{\phi}_{\ell^\infty(H^{k+1})}
\end{split}
\end{equation*}
and the Poincar\'e inequality,  
\begin{align}
 \mynorm{v}_{1}  \leq c \mynorm{\nabla v}, \quad \forall v \in H_0^1(\Omega).
\label{poincare}\end{align}

Now we show the estimate (\ref{eq:mainest3}). 
We return to the error equation (\ref{eq:errorEquation}).
Using \eqref{h2regularity} in place of (\ref{h1estimate}) in the estimate of $R_3^n$, we can evaluate (\ref{est:r3}) as 
\[
   \mynorm{ R_3^n}    
   \leq \frac{\cpro h^{k+1}}{\sqrt{\Delta t}} \mynorm{\henbi{\phi}{t}}_{L^2(t^{n-1},t^n;H^{k+1})}.
\]
From Lemma \ref{lemm:DouglasRussell}
we have
\begin{equation*}
	\normi{ R_4^n}_{H^{-1}(\Omega)} 
	\leq \cuone \mynorm{\proe^{n-1}}
	\leq \cuone h^{k+1} \normsob{\phi^{n-1}}{k+1}{2}.
\end{equation*}
Hence, it holds that 
\begin{equation*}
\begin{split}
	( R_4^n,e_h^n) &\leq \mynorm{ R_4^n}_{H^{-1}(\Omega)} \normsob{e_h^n}{1}{2}
	\leq \frac{\cuone}{\nu}h^{2(k+1)}\normsob{\phi^{n-1}}{k+1}{2}^2 + \frac{\nu}{2}\mynorm{\nabla e_h^n}^2,
\end{split}
\end{equation*}
where we have used the Poincar\'e inequality \eqref{poincare}.  
Using this inequality instead of $\frac{1}{4\eps_4}\mynorm{R_4^n}^2+ \eps_4\mynorm{e_h^n}^2$ in (\ref{eq:modifiedInterpole1})
and replacing the last term of (\ref{eq:triangle}) by $\cpro h^{k+1} \mynorm{\phi}_{\ell^\infty(H^{k+1})}$,
we obtain (\ref{eq:mainest3}).

\end{proof}
%
\section{Numerical results}\label{sec:numer}
We show numerical results in $d=2$.
We compare the conventional scheme (\jurai) with the present one (\zettai).
We use FreeFem++ \cite{FreeFemCite} for the triangulation of the domain.
Both $\pk1$- and $\pk2$-elements are used.
For \jurai \ we use the seven points quadrature formula of degree five \cite{HMS}. 
A relative error $E_X$ is defined by
	\begin{equation*}
		E_X \equiv \frac{\|\Pi_h^{(k)}\phi - \phi_h\|_{\ell^\infty(X)}} 		{\|\Pi_h^{(k)} \phi\|_{\ell^\infty(X)}},
	\end{equation*}
where
$\Pi_h^{(k)}$ is the Lagrange interpolation operator to the $\pk k$-finite element space and 
$X=L^2(\Omega)$ or $H_0^1(\Omega)$.
\begin{exam}[The rotating Gaussian hill \cite{RuiTabata2002}]\label{exam:rotating}
In (\ref{convdifeq}), $\Omega$ is an unit disk, and we set $T=2\pi, \nu=10^{-5}$, 
\begin{equation*}
	u(x,t) \equiv (-x_2,x_1), ~ f \equiv 0, ~ \phi^0 \equiv \phi_e(\cdot,0), 
\end{equation*}
where
\begin{equation*}
\begin{split}
 	&\phi_e(x,t) \equiv
 	\frac{\sigma}{\sigma + 4\nu t} \exp \left\{ -\frac{(\bar x_1(t) - x_{1,c})^2 +(\bar x_2(t) - x_{2,c})^2}{\sigma + 4\nu t} \right\}, \\
 	&(\bar x_1,\bar x_2)(t) \equiv (x_1 \cos t + x_2 \sin t,-x_1 \sin t + x_2 \cos t),\\
	&(x_{1,c}, x_{2,c}) \equiv (0.25,0), ~ \sigma \equiv 0.01. 
\end{split}
\end{equation*}
\end{exam}
\onefig{
	\includegraphics[width=50mm]{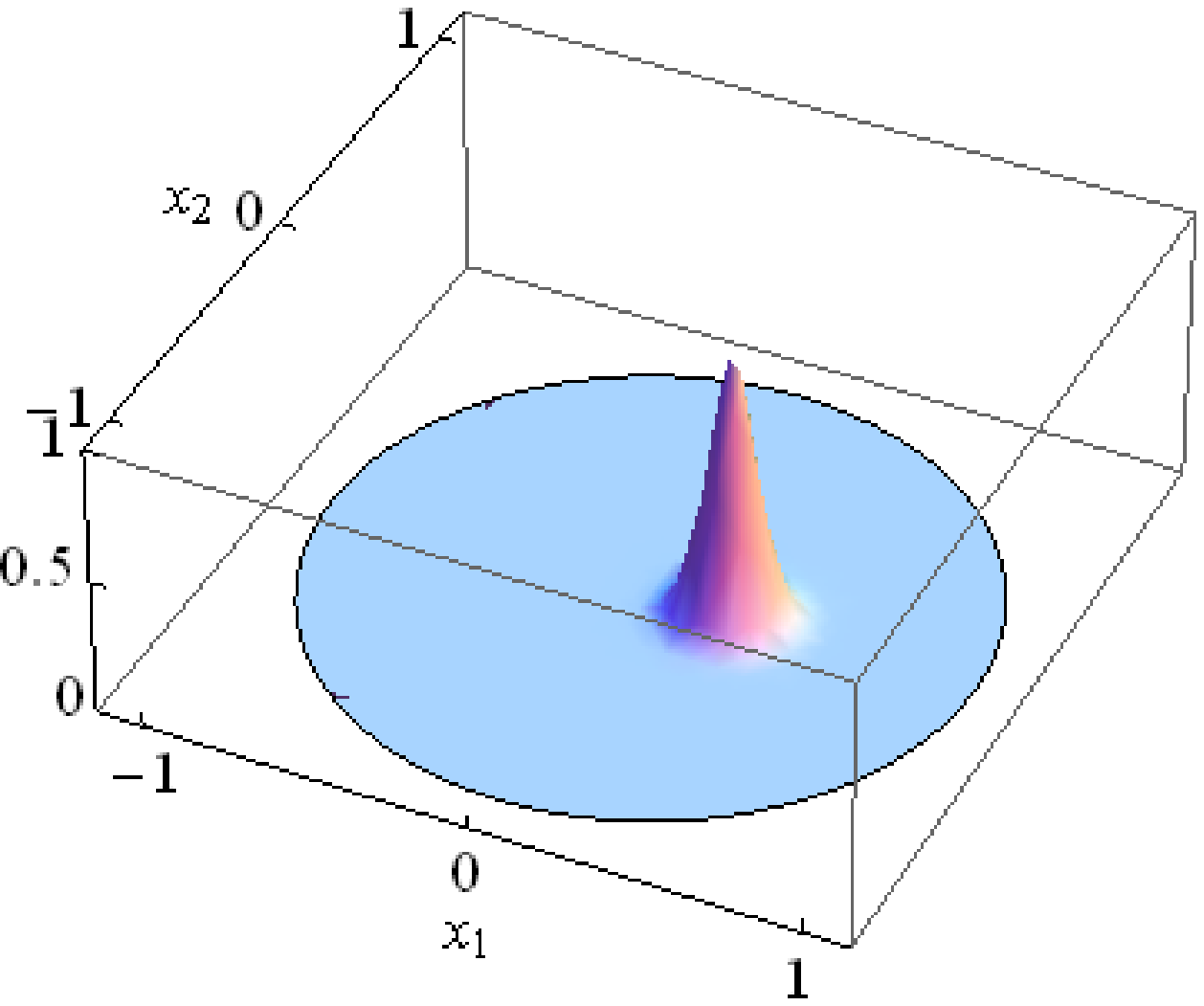}
	\quad
	\includegraphics[width=40mm]{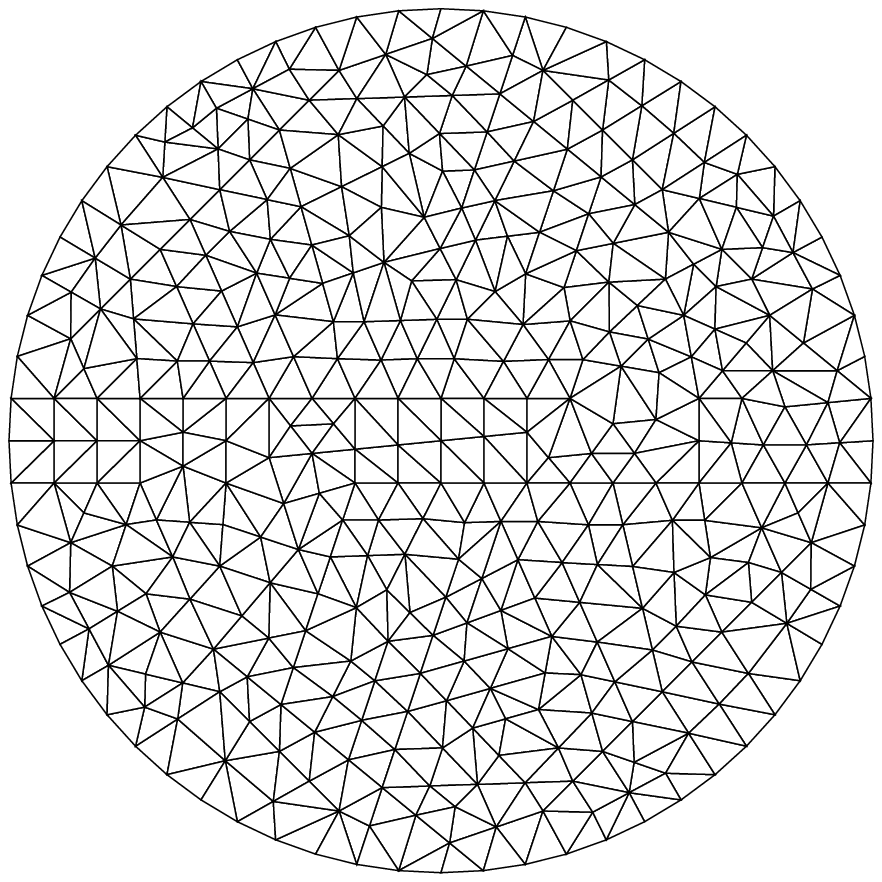} 
	\caption{
		The function $\phi_e (\cdot,0)$ (left) and the triangulation of $\bar \Omega$ for $N=64$ (right) in Example \ref{exam:rotating}
	}
	\label{fig:rhexactAndcirc64}
}
In this problem 
the identity $\Pi_h^{(1)}u=u$ holds.
This problem does not satisfy our setting because
$\Omega$ is not a polygon and $u\not=0$ on $\partial \Omega$.
The function $\phi_e$ in Fig. \ref{fig:rhexactAndcirc64} (left) satisfies (\ref{convdifeqa}) and (\ref{convdifeqc}) but does not satisfy the boundary condition (\ref{convdifeqb}). 
However, 
we may apply the schemes and treat $\phi_e$ as the solution since the value of $\phi_e$ on $\partial \Omega$ is almost equal to zero, less than $10^{-15}$, 
and we may neglect the effect of the boundary value and 
the term $\int_K(\phi_h^{n-1}\circ X_1^{n}) \psi_h dx$ and 
$\int_K(\phi_h^{n-1}\circ X_{1h}^{n})\psi_h dx$ 
on the element $K$ touching the boundary.

Let $N$ be the division number of the circle. 
We set $h\equiv 2\pi/N, N=32, 64, 128$ and $256$.
Figure \ref{fig:rhexactAndcirc64} (right) shows the triangulation of $\bar \Omega$ for $N=64$.
The time increment $\Delta t$ is set to be
$c_1 h $ and
$c_2 h^2 $ for $\pk1$-element ($c_1 = \frac{4}{5\pi}\nequal 0.255,~c_2 = \frac{64}{5\pi^2}\nequal 1.30$), 
$c_3 h^2 $ and
$c_4 h^3 $ for $\pk2$-element ($c_3 = \frac{128}{5\pi^2}\nequal 2.59,~c_4 = \frac{2048}{5\pi^3}\nequal 13.21$)
so that we can observe the convergence behavior of $O(h^k)$ for $E_{H_0^1}$, and $O(h^k)$ and $O(h^{k+1})$ for $E_{L^2}$
when $\pk k$-element is employed.

In the following figures we use the symbols shown in Table 1.  
Figure \ref{fig:rotp1p2e} shows the log-log graphs of $E_{L^2}$ and $E_{H_0^1}$ versus $h$.  
The left graph shows the results of $\pk1$-element and Tables \ref{table:rotp1} shows the values of them.
The convergence order of $E_{L^2}$ with $\Delta t=O(h)$ is less than 1 in \jurai ~(\markone) and more than 1 in \zettai~(\marktwo). 
The orders of $E_{L^2}$ with $\Delta t=O(h^2)$ are almost 2 for small $h$ in both schemes (\markthree, \markfour).
The convergence of $E_{H_0^1}$ is not observed in \jurai~(\markfive) while the order is almost 1 in \zettai~(\marksix).
The right graph of Fig. \ref{fig:rotp1p2e} shows the results of $\pk2$-element and Tables \ref{table:rotp2} shows the values of them.
The errors $E_{L^2}$ with $\Delta t=O(h^2)$
are too large at $N=128$ and $256$ 
to be plotted in the graph in \jurai ~(\markone)
while 
the convergence order is almost 2 in \zettai~(\marktwo).
The error $E_{L^2}$ with $\Delta t=O(h^3)$
is large at $N=128$, but it becomes small again at $N=256$ in \jurai~(\markthree). 
We will discuss the reason why 
such a behavior occurs 
in a forthcoming paper.
The order 
is 
greater than 2.5 
in \zettai~(\markfour).
The errors $E_{H_0^1}$ are too large at $N=128$ and $256$
to be plotted in the graph in \jurai~(\markfive) 
while we can observe the convergence of $E_{H_0^1}$ but the order is less than 2 in \zettai~(\marksix).
The errors of \zettai ~are smaller than those of \jurai ~in both cases of $\pk1$- and $\pk2$-element.
\begin{mytable}
\caption{Symbols used in Figs. \ref{fig:rotp1p2e}, \ref{fig:sinnu-2} and \ref{fig:sinnu-5} and Tables 
\ref{table:rotp1}--\ref{table:sinnu-5_p2}
}
\centering
\begin{tabular}{cccc}
	\hline
	$X$ & $\ell^\infty(L^2)$ & $\ell^\infty(L^2)$ & $\ell^\infty(H^1_0)$ \rule{0cm}{\columnheight}\\
	$\Delta t$ & $O(h^k)$ & $O(h^{k+1})$ & $O(h^k)$  \\
	\hline
	\hline
	\jurai 	& \markone & \markthree & \markfive  \rule{0cm}{\columnheight}\\
	\zettai & \marktwo & \markfour & \marksix  \rule{0cm}{\columnheight}\\
	\hline
\end{tabular}
\label{table:symbol}
\end{mytable}
%
\onefig{
  	\begin{overpic}[width=\figwidthbasic]
 	{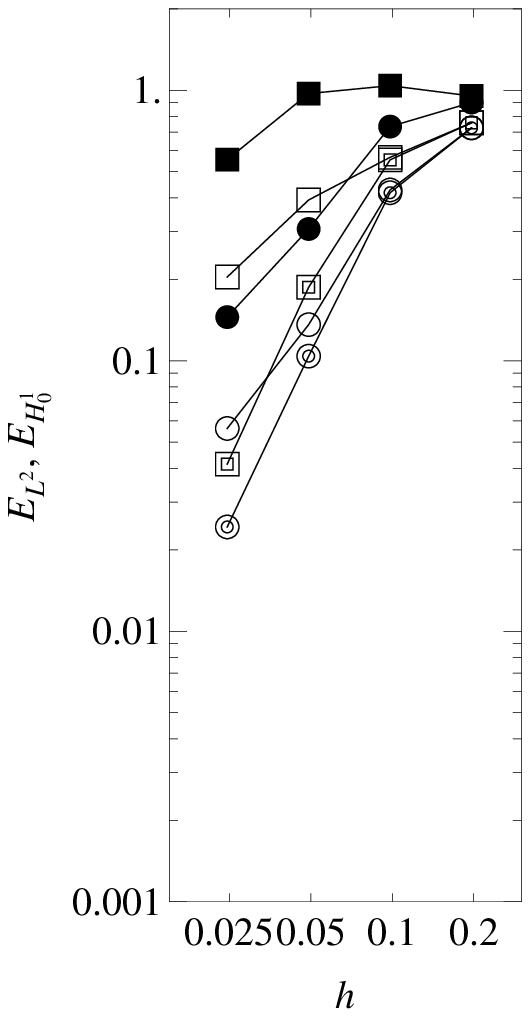}
  	\put(30,20){\includegraphics[width=10mm]{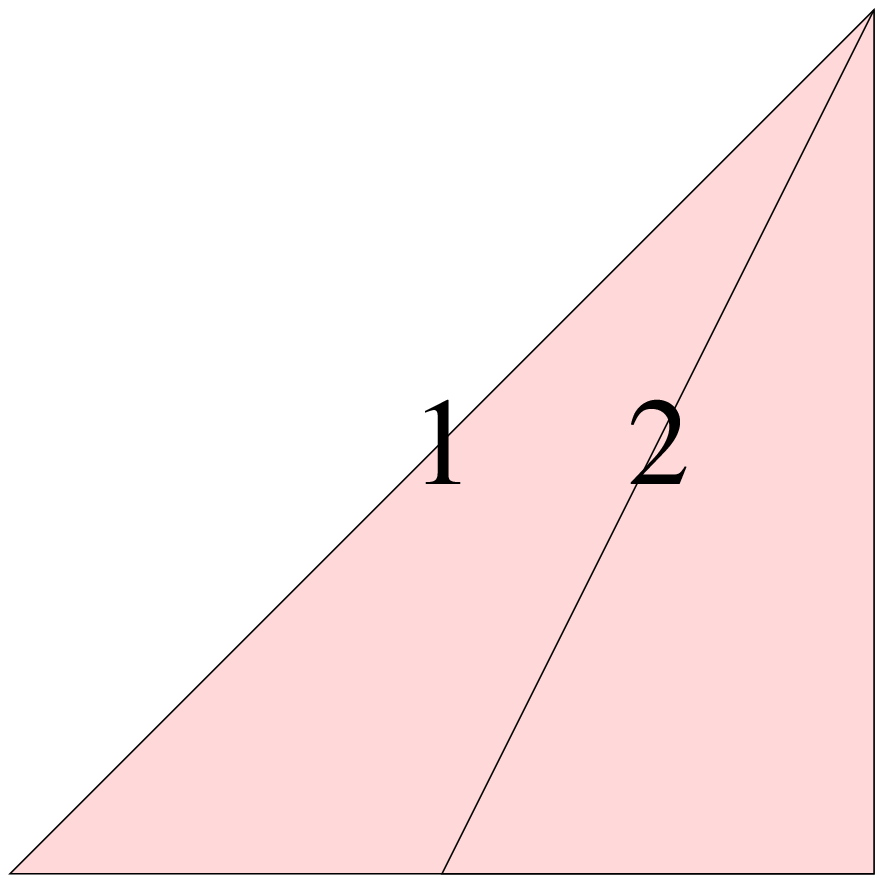} 
  	}
  	\end{overpic}
  	\begin{overpic}[width=\figwidthbasic]
 	{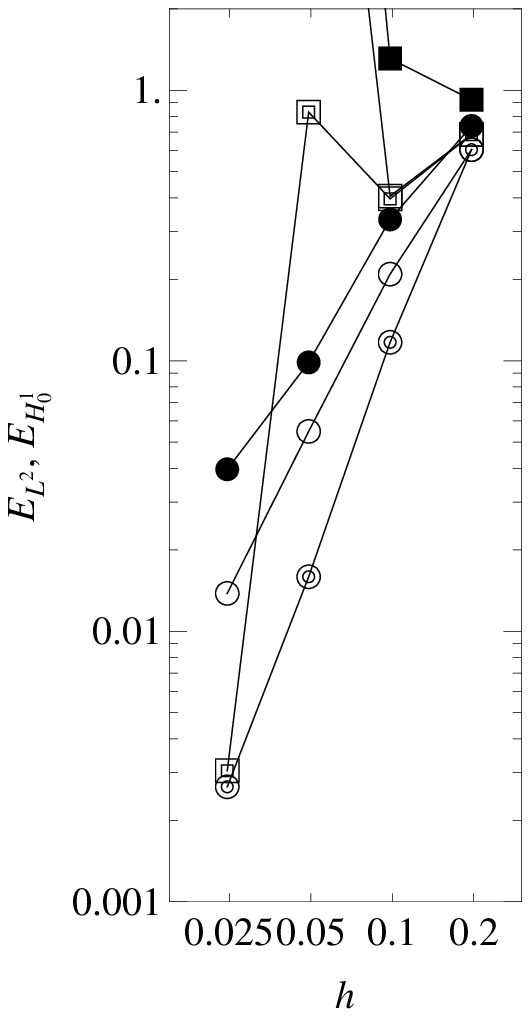}
  	\put(30,20){\includegraphics[width=8mm]{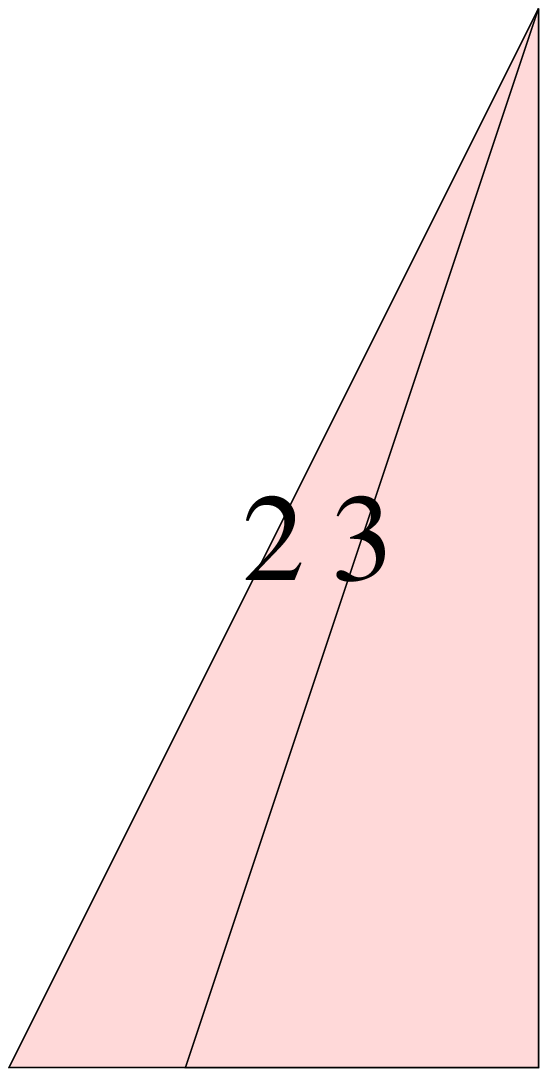}}
  	\end{overpic}
	\caption{Graphs of $E_{L^2}$ and $E_{H_0^1}$ versus $h$ in Example \ref{exam:rotating} by 
	$\pk k$-element. $k=1$ (left) and $k=2$ (right)
	}
	\label{fig:rotp1p2e}
}
\begin{mytable}
\caption{The values of errors and orders of the graph in Fig. \ref{fig:rotp1p2e} by $\pk 1$-element}
\label{table:rotp1}
\begin{tabular}{rlr@{\hspace{8mm}}lr@{\hspace{8mm}}lr}
\hline
$N$ & \markone & order & \markthree & order & \markfive & order \rule{0cm}{\columnheight}\\
\hline
\hline
32 & 7.58E-01 &  & 7.58E-01 &  & 9.52E-01 & \\
64 & 5.65E-01 & 0.42  & 5.53E-01 & 0.45  & 1.04E+00 & -0.13 \\
128 & 3.93E-01 & 0.52  & 1.87E-01 & 1.56  & 9.72E-01 & 0.10 \\
256 & 2.04E-01 & 0.95  & 4.15E-02 & 2.17  & 5.54E-01 & 0.81 \\
\hline
$N$ & \marktwo & order & \markfour & order & \marksix & order \rule{0cm}{\columnheight}\\
\hline
\hline
32 & 7.26E-01 &  & 7.26E-01 &  & 9.01E-01 & \\
64 & 4.28E-01 & 0.76  & 4.18E-01 & 0.80  & 7.34E-01 & 0.30 \\
128 & 1.36E-01 & 1.65  & 1.04E-01 & 2.01  & 3.07E-01 & 1.26 \\
256 & 5.62E-02 & 1.27  & 2.43E-02 & 2.10  & 1.45E-01 & 1.08 \\
\hline
\end{tabular}
\end{mytable}
\begin{mytable}
\caption{The values of errors and orders of the graph in Fig. \ref{fig:rotp1p2e} by $\pk 2$-element}
\label{table:rotp2}
\begin{tabular}{rlr@{\hspace{8mm}}lr@{\hspace{8mm}}lr}
\hline
$N$ & \markone & order & \markthree & order & \markfive & order \rule{0cm}{\columnheight}\\
\hline
\hline
32 & 6.86E-01 &  & 6.86E-01 &  & 9.22E-01 & \\
64 & 4.06E-01 & 0.76  & 3.97E-01 & 0.79  & 1.31E+00 & -0.51 \\
128 & 1.67E+02 & -8.68  & 8.30E-01 & -1.06  & 1.72E+03 & -10.36 \\
256 & 1.42E+27 & -82.81  & 3.05E-03 & 8.09  & 3.10E+28 & -83.90 \\
\hline
$N$ & \marktwo & order & \markfour & order & \marksix & order \rule{0cm}{\columnheight}\\
\hline
\hline
32 & 6.03E-01 &  & 6.03E-01 &  & 7.38E-01 & \\
64 & 2.09E-01 & 1.53  & 1.17E-01 & 2.37  & 3.33E-01 & 1.15 \\
128 & 5.48E-02 & 1.93  & 1.59E-02 & 2.88  & 9.86E-02 & 1.76 \\
256 & 1.38E-02 & 1.99  & 2.66E-03 & 2.58  & 3.97E-02 & 1.31 \\
\hline
\end{tabular}
\end{mytable}

Figure \ref{fig:profiles} shows the solutions $\phi_h^n$ for
$h=2\pi/64 \nequal 0.0982, \Delta t = 0.0065$, $n\Delta t \nequal 2\pi$.
In the case of $\pk1$-element, the solution of \jurai~is oscillatory
while that of \zettai ~is much better 
though a small ruggedness is observed. 
In the case of $\pk2$-element, the solution of \jurai ~is quite oscillatory 
while that of \zettai ~is stable.
\onefig{
\includegraphics[width=35mm]{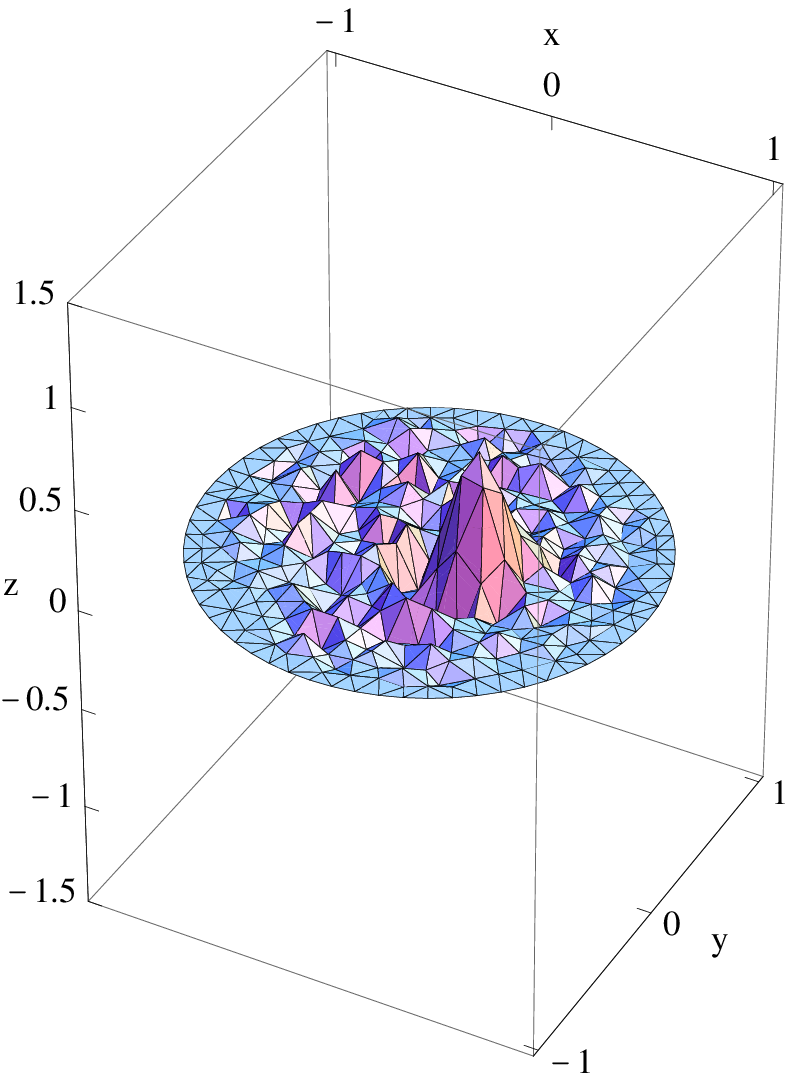}
\includegraphics[width=35mm]{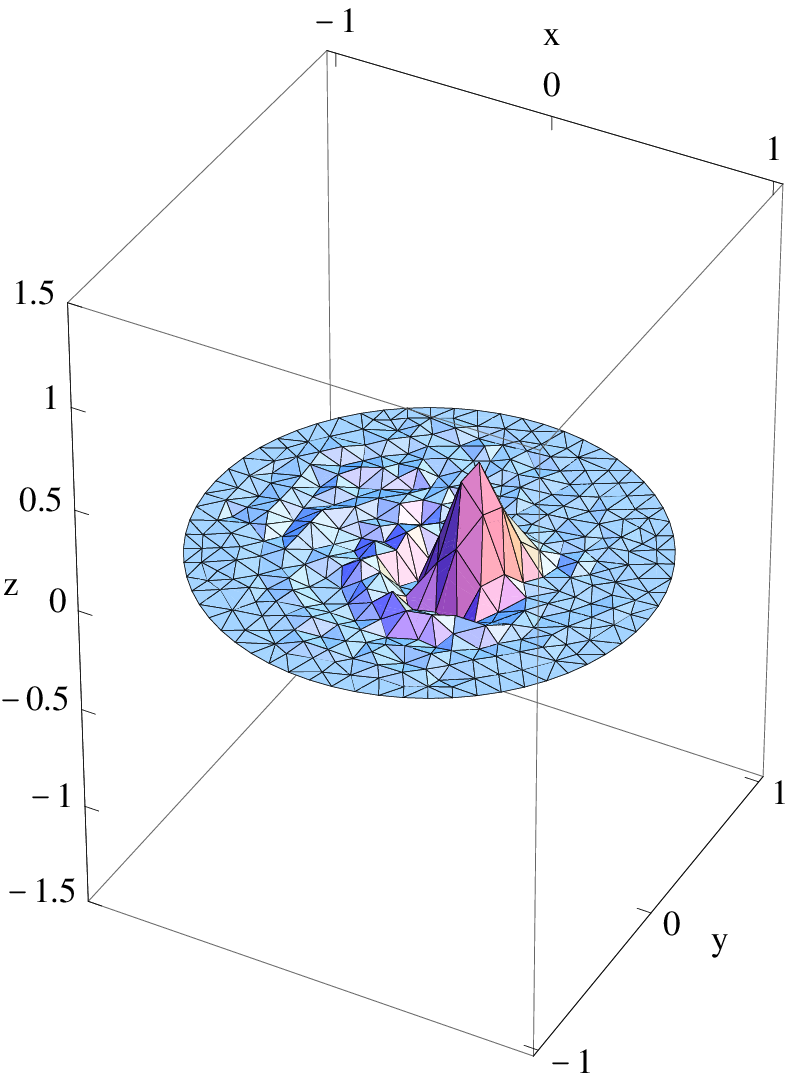}
\\
\includegraphics[width=35mm]{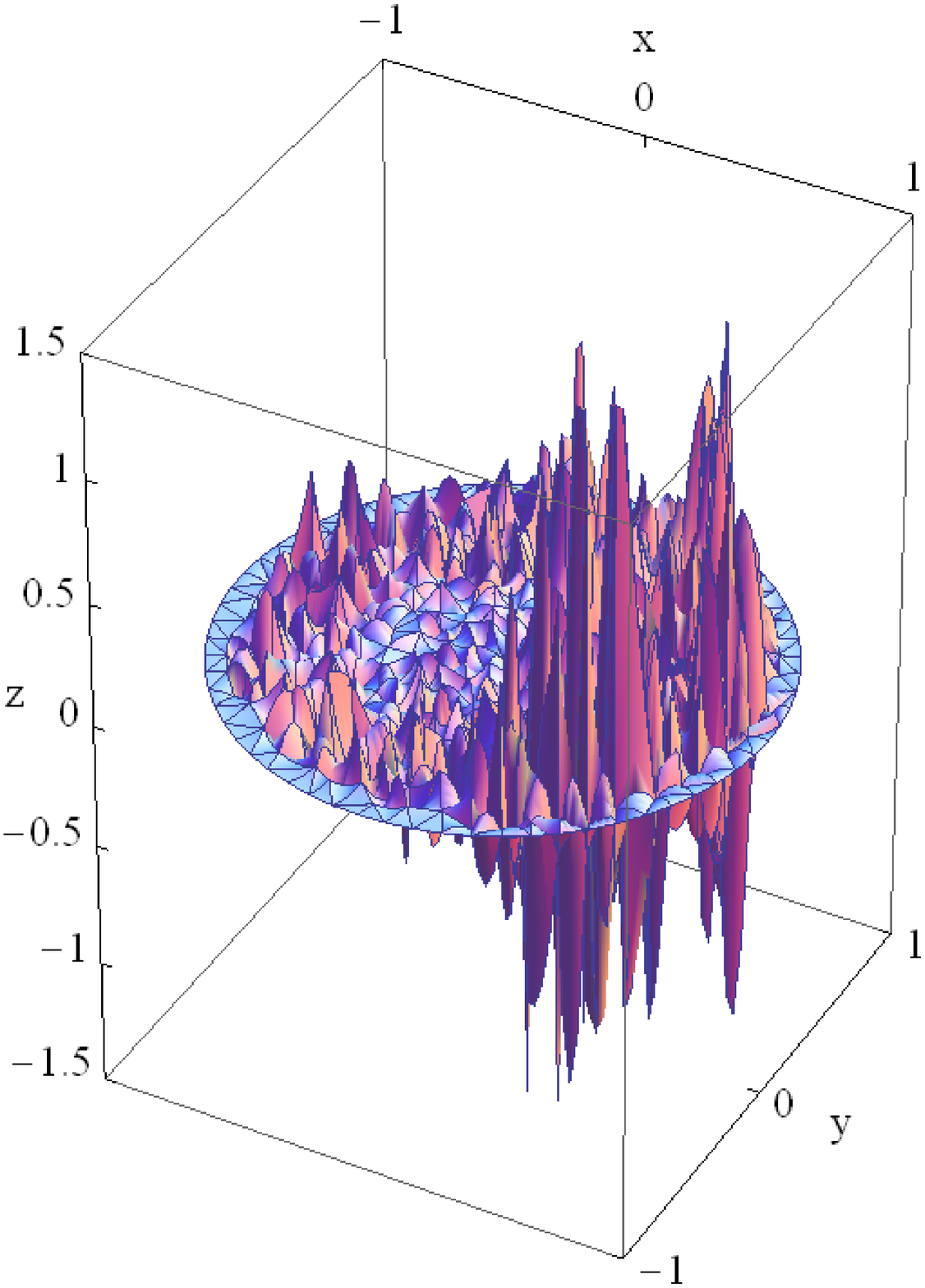}
\includegraphics[width=35mm]{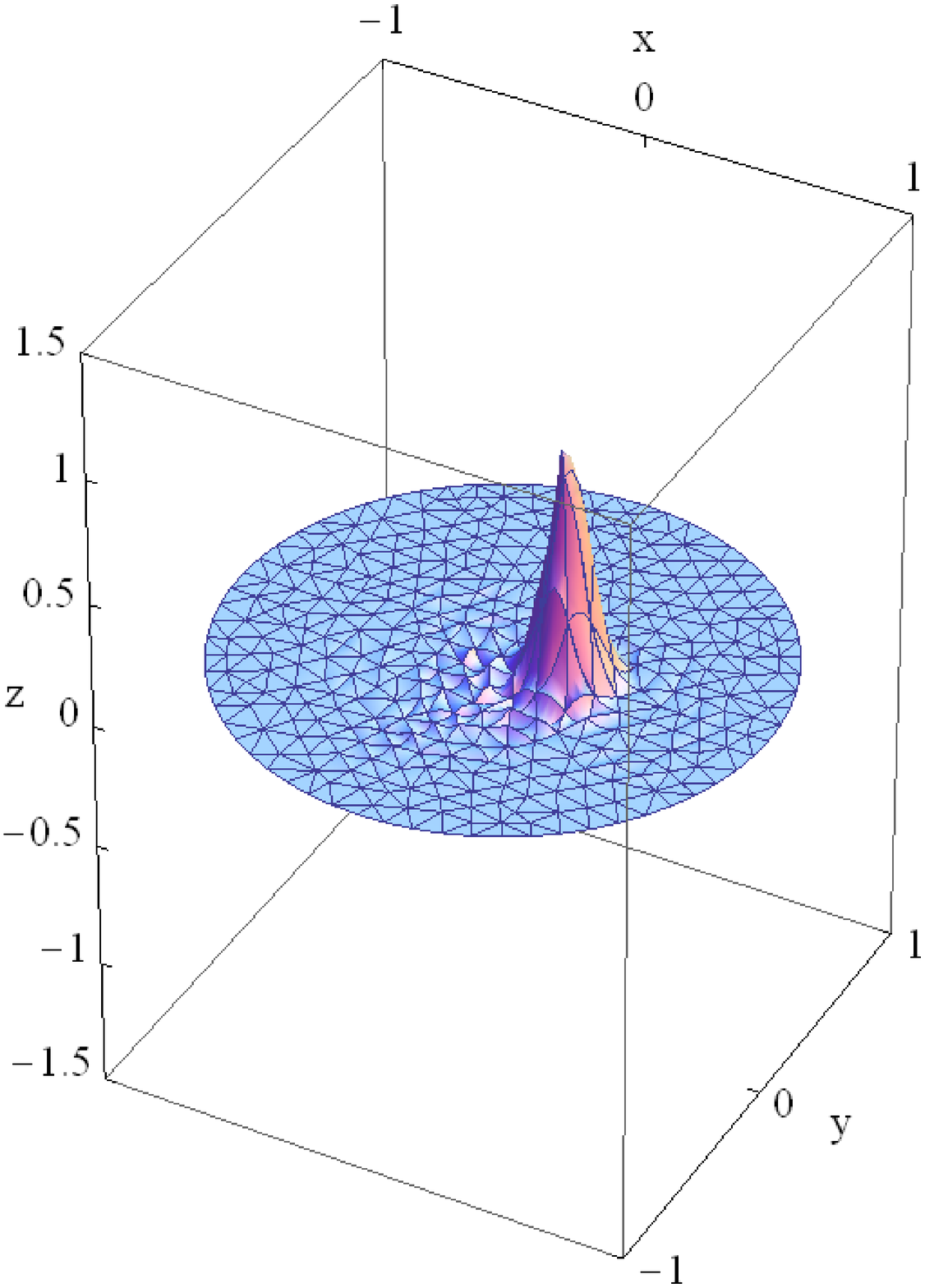}
\caption{
Solutions $\phi_h^n ~ (n\Delta t \nequal 2\pi)$ in Example \ref{exam:rotating} by 
\jurai (top left) and \zettai (top right) for $\pk1$-element, and by \jurai (bottom left) and \zettai (bottom right) for $\pk2$-element
}
\label{fig:profiles}
}
\begin{exam}%
\label{exam:sinsin}
In (\ref{convdifeq}), $\Omega$ is the square $(0,1)\times (0,1)$, and we set $T=1, \nu=10^{-2}$ and $10^{-5}$,
\begin{equation*}
\begin{split}
	&u(x,t) \equiv (\sin \pi x_1 \sin \pi x_2, \sin \pi x_1 \sin \pi x_2), ~
	f \equiv \henbi{\phi_e}{t}+u\cdot \nabla \phi_e - \nu \Delta  \phi_e,\\ 
	&\phi^0 \equiv \phi_e(\cdot,0), 
\end{split}
\end{equation*}
where
\begin{equation*}
	\phi_e(x,t) \equiv
		\cos(2\pi t) \sin^2(\pi x_1) \sin(2\pi x_2). 
\end{equation*}
\end{exam}
\onefig{
	\includegraphics[width=40mm]{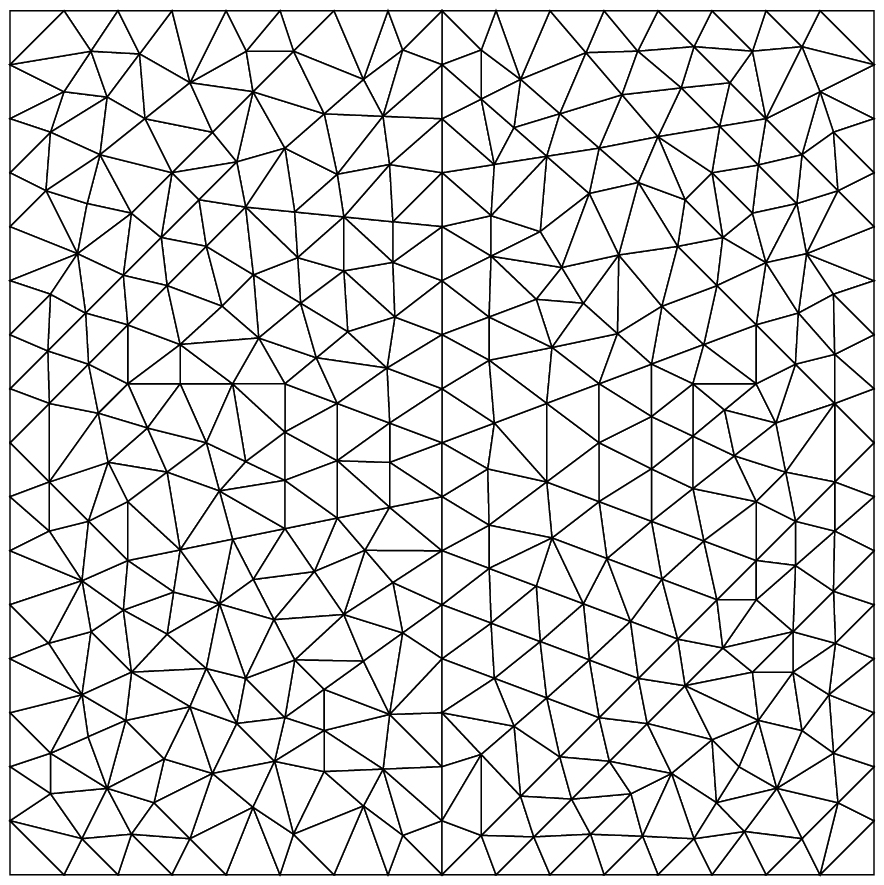} 
	\caption{The triangulation of $\bar \Omega$ for $N=16$ in Example \ref{exam:sinsin}}
	\label{fig:irsqori16}
}
In this problem, 
$\Pi_h^{(1)} u\not = u$.
Let $N$ be the division number of each side of $\bar \Omega$.
We set $h\equiv 1/N, N=8, 16, 32$ and $64$.
Figure \ref{fig:irsqori16} shows the triangulation of $\bar \Omega$ for $N=16$.
The time increment $\Delta t$ is set to be
$c_1 h $ and $c_2 h^2$ for $\pk1$-element $(c_1=0.125,c_2=1)$, 
$c_3 h^2$ and $c_4 h^3$ for $\pk2$-element $(c_3=1,c_4=5.12)$
so that 
we can observe the convergence behavior 
of $O(h^k)$ for $E_{H_0^1}$, and $O(h^k)$ and $O(h^{k+1})$ for $E_{L^2}$ when $\pk k$-element is employed.
\onefig{
  	\begin{overpic}[width=\figwidthbasic]{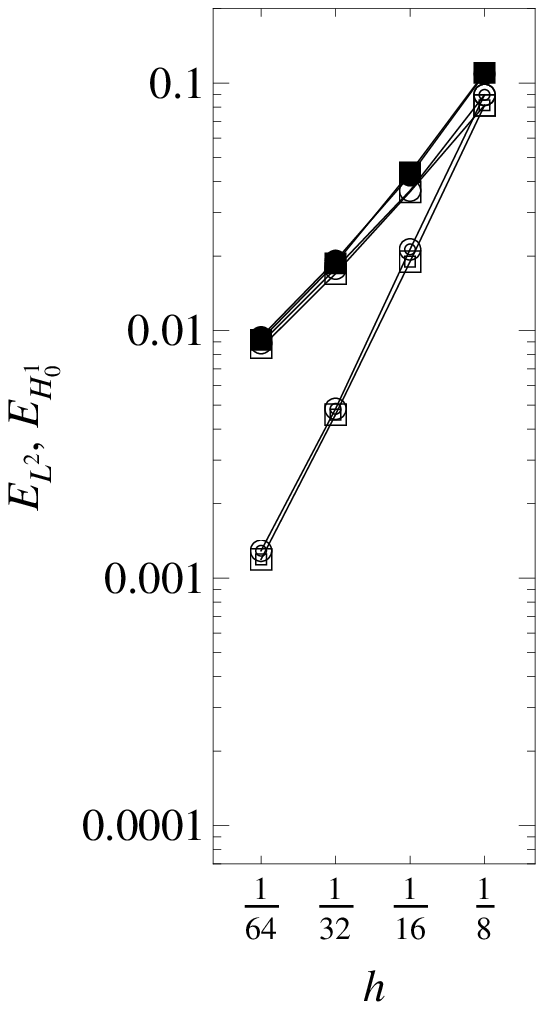}
  	\put(30,20){\includegraphics[width=10mm]{tri12.eps}}
  	\end{overpic}
	\begin{overpic}[width=\figwidthbasic]{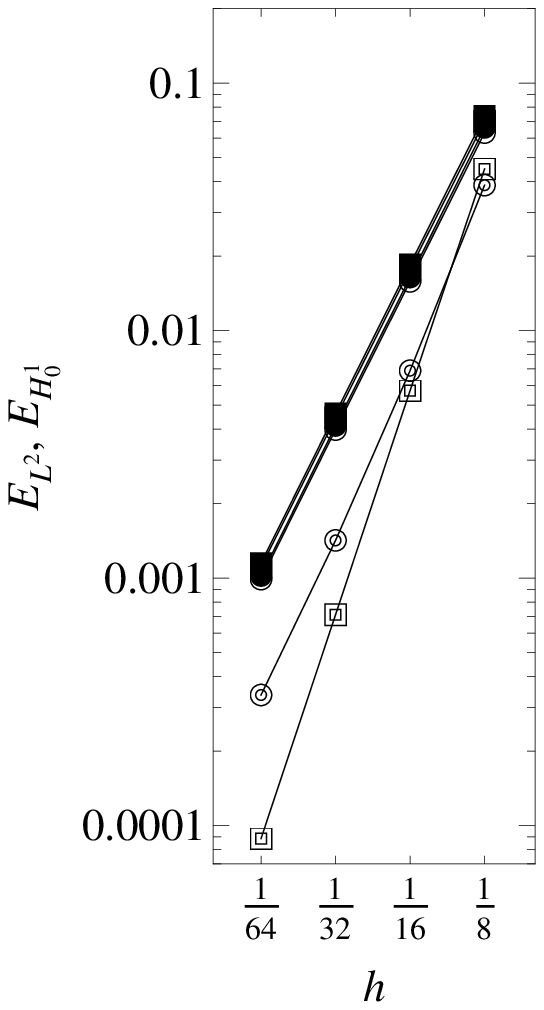}
	\put(35,15){\includegraphics[width=8mm]{tri23.eps}}
	\end{overpic}
	\caption{Graphs of $E_{L^2}$ and $E_{H_0^1}$ versus $h$ in Example \ref{exam:sinsin} with
	$\nu=10^{-2}$ by $\pk k$-element. $k=1$ (left) and $k=2$ (right)
	}
  	\label{fig:sinnu-2}
}
\begin{mytable}
\caption{The values of errors and orders of the graph in Fig. \ref{fig:sinnu-2} by $\pk 1$-element}
\label{table:sinnu-2_p1}
\begin{tabular}{rlr@{\hspace{8mm}}lr@{\hspace{8mm}}lr}
\hline
$N$ & \markone & order & \markthree & order & \markfive & order \rule{0cm}{\columnheight}\\
\hline
\hline
8 & 8.14E-02 &  & 8.14E-02 &  & 1.10E-01 & \\
16 & 3.64E-02 & 1.16  & 1.90E-02 & 2.10  & 4.36E-02 & 1.34 \\
32 & 1.70E-02 & 1.10  & 4.58E-03 & 2.05  & 1.87E-02 & 1.22 \\
64 & 8.53E-03 & 0.99  & 1.19E-03 & 1.94  & 9.18E-03 & 1.03 \\
\hline
$N$ & \marktwo & order & \markfour & order & \marksix & order \rule{0cm}{\columnheight}\\
\hline
\hline
8 & 8.97E-02 &  & 8.97E-02 &  & 1.09E-01 & \\
16 & 3.68E-02 & 1.29  & 2.13E-02 & 2.07  & 4.23E-02 & 1.37 \\
32 & 1.78E-02 & 1.05  & 4.83E-03 & 2.14  & 1.92E-02 & 1.14 \\
64 & 8.90E-03 & 1.00  & 1.29E-03 & 1.90  & 9.43E-03 & 1.03 \\
\hline
\end{tabular}
\end{mytable}
\begin{mytable}
\caption{The values of errors and orders of the graph in Fig. \ref{fig:sinnu-2} by $\pk 2$-element}
\label{table:sinnu-2_p2}
\begin{tabular}{rlr@{\hspace{8mm}}lr@{\hspace{8mm}}lr}
\hline
$N$ & \markone & order & \markthree & order & \markfive & order \rule{0cm}{\columnheight}\\
\hline
\hline
8 & 7.01E-02 &  & 4.50E-02 &  & 7.37E-02 & \\
16 & 1.77E-02 & 1.99  & 5.72E-03 & 2.98  & 1.85E-02 & 1.99 \\
32 & 4.44E-03 & 2.00  & 7.11E-04 & 3.01  & 4.63E-03 & 2.00 \\
64 & 1.11E-03 & 2.00  & 8.86E-05 & 3.00  & 1.15E-03 & 2.01 \\
\hline
$N$ & \marktwo & order & \markfour & order & \marksix & order \rule{0cm}{\columnheight}\\
\hline
\hline
8 & 6.31E-02 &  & 3.87E-02 &  & 6.60E-02 & \\
16 & 1.58E-02 & 2.00  & 6.89E-03 & 2.49  & 1.64E-02 & 2.01 \\
32 & 3.98E-03 & 1.99  & 1.42E-03 & 2.28  & 4.12E-03 & 1.99 \\
64 & 9.90E-04 & 2.01  & 3.37E-04 & 2.08  & 1.02E-03 & 2.01 \\
\hline
\end{tabular}
\end{mytable}
\onefig{
   	\begin{overpic}[width=\figwidthbasic]{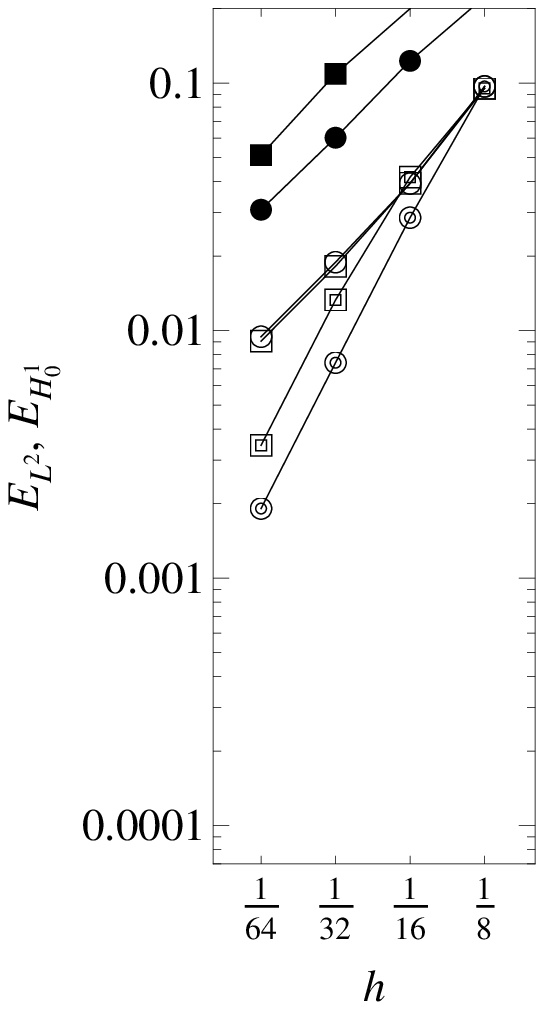}
  	\put(30,20){\includegraphics[width=10mm]{tri12.eps}}
  	\end{overpic}
	\begin{overpic}[width=\figwidthbasic]{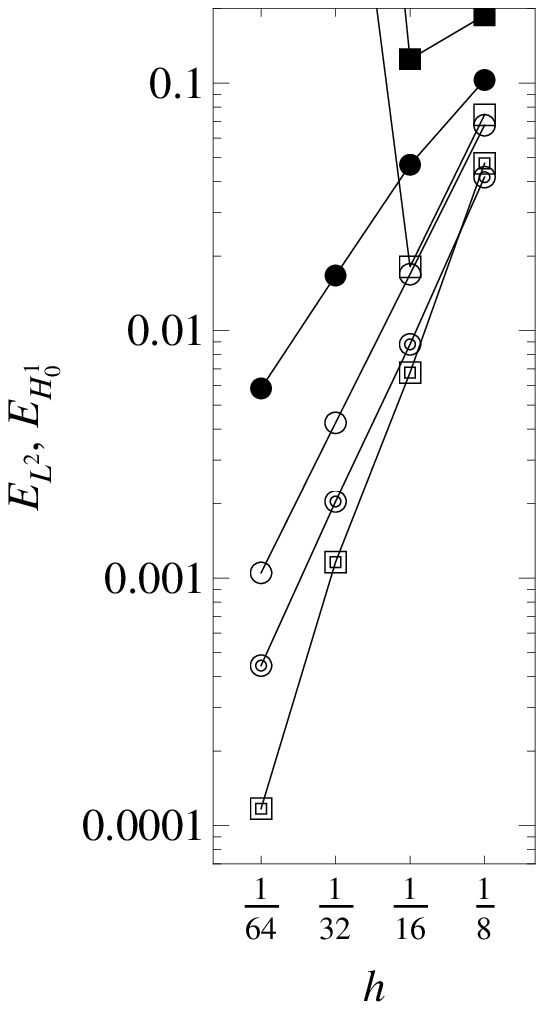}
	\put(35,15){\includegraphics[width=8mm]{tri23.eps}}
	\end{overpic}
	\caption{Graphs of $E_{L^2}$ and $E_{H_0^1}$ versus $h$ in Example \ref{exam:sinsin} with
	$\nu=10^{-5}$ by $\pk k$-element. $k=1$ (left) and $k=2$ (right)
	}
	\label{fig:sinnu-5}
}
\begin{mytable}
\caption{The values of errors and orders of the graph in Fig. \ref{fig:sinnu-5} by $\pk 1$-element}
\label{table:sinnu-5_p1}
\begin{tabular}{rlr@{\hspace{8mm}}lr@{\hspace{8mm}}lr}
\hline
$N$ & \markone & order & \markthree & order & \markfive & order \rule{0cm}{\columnheight}\\
\hline
\hline
8 & 9.53E-02 &  & 9.53E-02 &  & 2.90E-01 & \\
16 & 3.94E-02 & 1.27  & 4.17E-02 & 1.19  & 2.00E-01 & 0.54 \\
32 & 1.82E-02 & 1.11  & 1.33E-02 & 1.65  & 1.09E-01 & 0.88 \\
64 & 9.07E-03 & 1.00  & 3.44E-03 & 1.95  & 5.12E-02 & 1.09 \\
\hline
$N$ & \marktwo & order & \markfour & order & \marksix & order \rule{0cm}{\columnheight}\\
\hline
\hline
8 & 9.70E-02 &  & 9.70E-02 &  & 2.23E-01 & \\
16 & 3.93E-02 & 1.30  & 2.86E-02 & 1.76  & 1.23E-01 & 0.86 \\
32 & 1.89E-02 & 1.06  & 7.42E-03 & 1.95  & 6.02E-02 & 1.03 \\
64 & 9.45E-03 & 1.00  & 1.91E-03 & 1.96  & 3.08E-02 & 0.97 \\
\hline
\end{tabular}
\end{mytable}
\begin{mytable}
\caption{The values of errors and orders of the graph in Fig. \ref{fig:sinnu-5} by $\pk 2$-element}
\label{table:sinnu-5_p2}
\begin{tabular}{rlr@{\hspace{8mm}}lr@{\hspace{8mm}}lr}
\hline
$N$ & \markone & order & \markthree & order & \markfive & order \rule{0cm}{\columnheight}\\
\hline
\hline
8 & 7.45E-02 &  & 4.74E-02 &  & 1.88E-01 & \\
16 & 1.81E-02 & 2.04  & 6.77E-03 & 2.81  & 1.25E-01 & 0.59 \\
32 & 3.94E+00 & -7.77  & 1.16E-03 & 2.55  & 1.22E+02 & -9.93 \\
64 & 1.10E+00 & 1.84  & 1.17E-04 & 3.31  & 8.17E+01 & 0.58 \\
\hline
$N$ & \marktwo & order & \markfour & order & \marksix & order \rule{0cm}{\columnheight}\\
\hline
\hline
8 & 6.76E-02 &  & 4.17E-02 &  & 1.03E-01 & \\
16 & 1.69E-02 & 2.00  & 8.80E-03 & 2.24  & 4.68E-02 & 1.14 \\
32 & 4.24E-03 & 1.99  & 2.04E-03 & 2.11  & 1.67E-02 & 1.49 \\
64 & 1.05E-03 & 2.01  & 4.43E-04 & 2.20  & 5.84E-03 & 1.52 \\
\hline
\end{tabular}
\end{mytable}

Figure \ref{fig:sinnu-2} shows the log-log graphs of $E_{L^2}$ and $E_{H_0^1}$ versus $h$ with $\nu=10^{-2}$.  
The left graph shows the results of $\pk1$-element and Table \ref{table:sinnu-2_p1} shows the values of them.
The convergence orders of $E_{L^2}$ with $\Delta t=O(h)$ are almost 1 in both schemes (\markone, \marktwo).
The orders of $E_{L^2}$ with $\Delta t=O(h^2)$ are almost $2$ in both schemes (\markthree, \markfour).
The orders of $E_{H_0^1}$ are almost 1 in both schemes (\markfive, \marksix).
The right graph of Fig. \ref{fig:sinnu-2} shows the results of $\pk 2$-element and Table \ref{table:sinnu-2_p2} shows the values of them.
The convergence orders of $E_{L^2}$ with $\Delta t=O(h^2)$ are almost 2 in both schemes (\markone, \marktwo).
The order of $E_{L^2}$ with $\Delta t=O(h^3)$ is almost $3$ in \jurai~(\markthree) and almost $2$ in \zettai~(\markfour).
The orders of $E_{H_0^1}$ are almost 2 in both schemes (\markfive, \marksix).
These results are consistent with the theoretical 
ones of \zettai, $E_{L^2}=O(\Delta t + h^2 + h^{k+1})$ and $E_{H_0^1}=O(\Delta t + h^2 + h^{k})$.

Figure \ref{fig:sinnu-5} shows the log-log graphs of $E_{L^2}$ and $E_{H_0^1}$ versus $h$ with $\nu=10^{-5}$. 
The left graph shows the results of $\pk1$-element and Table \ref{table:sinnu-5_p1} shows the values of them.  
The convergence orders of $E_{L^2}$ with $\Delta t=O(h)$ are almost $1$ in both schemes (\markone, \marktwo). 
The orders of $E_{L^2}$ with $\Delta t=O(h^2)$ are almost $2$ for small $h$ in both schemes (\markthree, \markfour).
The orders of $E_{H_0^1}$ are almost $1$ in both schemes (\markfive, \marksix). 
The right graph of Fig. \ref{fig:sinnu-5} shows the results of $\pk2$-element and Table \ref{table:sinnu-5_p2} shows the values of them.  
The errors $E_{L^2}$ with $\Delta t=O(h^2)$ are too large at $N=32$ and $64$ to be plotted in the graph in \jurai \ (\markone) while the convergence order is almost $2$ in \zettai \ (\marktwo).
The order $E_{L^2}$ with $\Delta t=O(h^3)$ is almost $3$ for small $h$ in \jurai \ (\markthree) and almost $2$ in \zettai \ (\markfour).
The errors $E_{H_0^1}$ are too large at $N=32$ and $64$ to be plotted in the graph in \jurai~(\markfive) while we can observe the convergence but the order is less than $2$ in \zettai~(\marksix).
In order to obtain the theoretical convergence order $O(h^2)$, it seems that finer mesh will be necessary.
%
\section{Conclusions}\label{sec:conclusion}
We have presented a genuinely stable Lagrange--Galerkin scheme for convection-diffusion problems.
In the scheme locally linearized velocities are used and the integration is executed exactly without numerical quadrature.
For the $\pk k$-element we have shown error estimates 
of $O(\Delta t + h^2 + h^{k+1})$
in $\ell^\infty(L^2)$-norm and 
of $O(\Delta t + h^2 + h^k)$
in $\ell^\infty (H^1)$-norm.  
We have also obtained error estimate, $c(\Delta t + h^2 + h^k)$
in $\ell^\infty(L^2)$-norm, 
where the coefficient $c$ is dependent on the exact solution $\phi$ but independent of the diffusion constant $\nu$.   
Numerical results have reflected these estimates.
The extension to the Navier--Stokes equations will be discussed in a forthcoming paper.
\begin{acknowledgements}
The first author was supported by JSPS (Japan Society for the Promotion of Science) under Grants-in-Aid for Scientific Research (C)No. 25400212 and (S)No. 24224004 and 
under the Japanese-German Graduate Externship (Mathematical Fluid Dynamics) and by Waseda University under Project research, Spectral analysis and its application to the stability theory of the Navier-Stokes equations of Research Institute for Science and Engineering.
The second author was supported by JSPS 
 under Grant-in-Aid for JSPS Fellows No. 26$\cdot$964.
\end{acknowledgements}
\bibliographystyle{plain}

\end{document}